\documentclass[11pt]{article}
\usepackage[margin=1.1in]{geometry}
\usepackage{amsmath,amssymb,amsthm,amsfonts}
\usepackage{mathtools}
\usepackage{graphicx}
\usepackage{booktabs}
\usepackage{tikz}
\usetikzlibrary{shapes.geometric,positioning}
\usepackage{algorithm}
\usepackage{algorithmic}
\usepackage[colorlinks=true,linkcolor=blue!50!black,citecolor=blue!50!black,urlcolor=blue!50!black]{hyperref}

\theoremstyle{plain}
\newtheorem{theorem}{Theorem}[section]
\newtheorem{lemma}[theorem]{Lemma}
\newtheorem{proposition}[theorem]{Proposition}

\theoremstyle{definition}

\theoremstyle{remark}

\newcommand{\email}[1]{\texttt{#1}}

\newcommand{\sphere}{\mathbb{S}^2}
\newcommand{\R}{\mathbb{R}}
\newcommand{\dA}{\,dA}
\newcommand{\gs}{g_{\scriptscriptstyle S}}

\newcommand{\osc}{\operatorname{osc}}

\newcommand{\Lqw}[1]{L^{#1}_{w}}
\newcommand{\norm}[1]{\lVert #1 \rVert}
\newcommand{\abs}[1]{\lvert #1 \rvert}
\newcommand{\tri}{\triangle}

\newcommand{\range}{\operatorname{range}}
\newcommand{\rank}{\operatorname{rank}}

\title{Conformal Spherical Splines for the Laplace--Beltrami Operator on
Genus-Zero Surfaces: Construction, Algorithm, and Experiments}
\author{Shelvean Kapita\\[2pt]
\small Department of Mathematics, Texas A\&M University, College Station, TX 77843, USA\\
\small \email{kapita@tamu.edu}}
\date{}

\begin{document}
\maketitle

\begin{abstract}
Every smooth closed Riemannian surface of genus zero is isometric to the
unit sphere carrying a conformal metric, the round metric multiplied by
a positive factor. This paper solves the Poisson equation and the
eigenvalue problem of the Laplace--Beltrami operator of such a metric
with smooth spherical splines, taking the conformal factor as the only
description of the geometry. Because the Dirichlet energy is conformally
invariant in two dimensions, the stiffness matrix is the round-sphere
matrix for every conformal metric, and the factor enters only the load
vector, the mean-value constraint and the weighted mass matrix. A
surface embedded in space with a conformal parametrization, a surface
of revolution, and a surface given only as a triangle mesh are
instances of the same problem that differ in how the factor is
obtained; the first two are treated here, and the factor is checked
against an area identity before any equation is discretized. The
unknown is a spline on a fixed spherical triangulation in broken
Bernstein--B\'ezier form, with $C^r$ smoothness imposed algebraically
through edge functionals and the conforming space realized by a
null-space matrix computed once and reused for every metric. Each step
of the computation is stated with the formulas it needs. Experiments on
the round metric, a prolate spheroid, a dumbbell of revolution, an
Evans--Fung red-blood-cell profile and a prescribed factor with no
embedding show the rates $d+1$ in $L^2$ and $d$ in energy for spline
degree $d$, the rate $2d$ for eigenvalues with reproduction of the
even-order eigenvalues of the round metric up to degree $d$ to
roundoff, and, against parametric surface finite elements on the same
triangulations, smaller errors at comparable numbers of unknowns for
degrees four and six.
\end{abstract}

\noindent\textbf{Keywords.} Laplace--Beltrami operator, Poisson equation on surfaces, spherical splines,
conformal parametrization, Bernstein--B\'ezier methods, genus-zero surfaces

\noindent\textbf{MSC 2020.} 65N30, 65N25, 65D07, 41A15, 58J50

\section{Introduction}\label{sec:intro}

Let $\lambda$ be a smooth function on the unit sphere $\sphere$ and let
$\bar g = e^{2\lambda}\gs$ be the conformal metric it defines, $\gs$
being the round metric. We write $M$ for the Riemannian surface
$(\sphere, \bar g)$, $\Delta_M$ for its Laplace--Beltrami operator and
$dA_M = e^{2\lambda}\dA$ for its area element. For a function $f$ on
$M$ with zero mean, the Poisson equation
\begin{equation}\label{eq:poissonM}
  -\Delta_M u = f \quad\text{on } M, \qquad \int_M u \dA_M = 0,
\end{equation}
and the eigenvalue problem $-\Delta_M u = \mu u$ are the two model
problems for second-order elliptic computation on a surface, and they
share their matrices. By the uniformization theorem \cite{Jost06} every
smooth closed Riemannian surface of genus zero is isometric to such an
$M$ for some $\lambda$, so \eqref{eq:poissonM} covers the
Laplace--Beltrami operator of every smooth closed genus-zero surface,
and the isometry is what a conformal parametrization of an embedded
surface provides. The question of this paper is practical: given
$\lambda$, how does one turn either problem into a matrix problem for a
smooth high-order spline on the sphere, and what has to be computed at
each step? The answer has three parts, and the paper is organized
around them.

First, the geometry. The whole of it is the conformal factor
$w = e^{2\lambda}$, one positive function on the sphere. How $w$ is
obtained depends on the instance: for a surface in $\R^3$ with a
conformal parametrization $\Phi\colon\sphere\to M$, $w$ is the area
factor of $\Phi$; for a surface of revolution, $w$ comes from a
one-dimensional nonlinear equation, the Mercator matching, and is
checked against a closed-form area; a prescribed $\lambda$ needs no
surface at all; and a surface given only as a triangle mesh requires
the map to be computed, which is deferred to a separate paper. All
four instances feed the same computation (Section~\ref{sec:extension}).

Second, the approximation space. The unknown is a spherical spline of
degree $d$ and smoothness $r$ on a fixed spherical triangulation
\cite{AlfNS96b, LaiSchumaker07}. We write it in broken
Bernstein--B\'ezier form \cite{Kapita26BCC}: every spherical triangle
keeps its own $\binom{d+2}{2}$ coefficients, $C^r$ smoothness across an
edge is a set of linear functionals whose only geometric input is the
spherical barycentric coordinates of the opposite vertex, the conforming
space is the kernel of the stacked smoothness matrix, and a null-space
matrix computed once turns the constrained problem into a symmetric
positive definite system. Nothing in this part depends on $\lambda$
(Section~\ref{sec:background}).

Third, the equation. Because the Dirichlet energy is conformally
invariant in two dimensions, the weak form of \eqref{eq:poissonM} in
the round metric reads
\begin{equation}\label{eq:poissonS-intro}
  \int_{\sphere} \nabla u \cdot \nabla v \dA
  = \int_{\sphere} w\, f\, v \dA
  \quad\text{for all } v, \qquad \int_{\sphere} w\, u \dA = 0,
\end{equation}
with the round gradient and area element. The stiffness matrix is
therefore the round-sphere matrix for every conformal metric, assembled
once, and $\lambda$ enters the load vector and the mean constraint
through the single factor $w$ at the quadrature points
(Section~\ref{sec:poisson}); the eigenvalue problem is the generalized
eigenproblem of the same stiffness matrix and the weighted mass
matrix. Section~\ref{sec:algorithm} lists the computation step by
step with the formulas each step evaluates.

The contribution is the formulation in terms of the conformal metric
and what it makes computable, rather than a new spline space: the
conformal metric on the sphere as the object discretized, with embedded
surfaces, surfaces of revolution, prescribed factors and meshed
surfaces as instances that differ only in how $w$ is supplied; the
geometry carried by one scalar weight without approximation whenever
$\lambda$ is known exactly, so that the discretization itself
introduces no geometric consistency error; the metric-independent
stiffness matrix and the reusable smoothness and null-space matrices;
the eigenvalue formulation with its parity reproduction for the round
metric; and a computational comparison with parametric surface finite
elements. For the round metric the linear system is that of Baramidze
and Lai \cite{BarLai06}.

The experiments (Section~\ref{sec:numerics}) use five instances: the
round metric, a prolate spheroid, a strongly necked dumbbell of
revolution, the Evans--Fung red-blood-cell profile in physical units,
and a prescribed factor with no embedding. On all of them, the spline
solution of the Poisson equation
converges at the rates $d+1$ in $L^2(M)$ and $d$ in the $H^1(M)$
seminorm expected from the approximation theory of spherical splines
\cite[Ch.~14]{LaiSchumaker07}, and the eigenvalues converge at the rate $2d$
with reproduction of the even-order eigenvalues of the round metric to
roundoff. The analysis behind the rates, and surfaces given only as
triangle meshes, for which $\lambda$ must itself be computed, are
deferred to separate papers; here the emphasis is on what is computed
and how.

Throughout, $\sphere$ is the unit sphere with round metric $\gs$,
surface measure $\dA$ and tangential gradient $\nabla$; $W^{k,q}$ are
Sobolev spaces with norm $\norm{\cdot}_{k,q,\Omega}$ and top-order
seminorm $\abs{\cdot}_{k,q,\Omega}$; eigenvalues are $\mu$ and
$\lambda$ is reserved for the logarithm of the conformal factor;
references of the form
Theorem~13.30 are to \cite{LaiSchumaker07}.

\section{Related work}\label{sec:related}

\emph{Spherical splines.} The foundational papers are
\cite{AlfNS95, AlfNS96a, AlfNS96b, AlfNS96c}, which introduced circular
and spherical Bernstein--B\'ezier polynomials, established dimension and
local bases for homogeneous spline spaces, and applied the construction
to fitting scattered data on spheres and sphere-like surfaces. The
approximation theory and the quasi-interpolation machinery are due to
Neamtu and Schumaker \cite{NeamtuSchumaker04},
\cite[Ch.~14]{LaiSchumaker07}. Spherical splines have been used
extensively for data fitting, by minimal energy and least squares
\cite{BarL05, BarLaiShum06, BarLai11} and with shape constraints on the
Bernstein--B\'ezier coefficients \cite{BarLai18}; such pointwise
constraints transfer unchanged to genus-zero surfaces through the
parametrization of this paper.

\emph{Spherical splines for PDEs, and the sphere-like setting of
Baramidze and Lai.} The closest prior work is that of Baramidze and
Lai \cite{BarLai06}: a second-order elliptic equation of
Laplace--Beltrami type on the unit sphere is solved by the Galerkin
method in $S^r_d(\tri)$, with existence, uniqueness and convergence
proved, the smoothness conditions enforced through the constrained
least-squares and augmented Lagrangian machinery of
\cite{AwanouLaiWenston06}, and numerical rates reported;
\cite{BarLaiShum06} introduces nonhomogeneous spherical splines,
restrictions of all trivariate polynomials of degree at most $d$, so
that constants are reproduced for every degree. In the same
constrained-coefficient framework, Lai and Lee \cite{LaiLee22} replace
the Galerkin form by collocation, the strong form imposed at interior
collocation points and the smoothness conditions imposed as constraints
on the coefficients, with the overdetermined constrained system solved
by least squares, convergence proved for domains of positive reach, and
the approach extended to the three-dimensional Monge--Amp\`ere
equation \cite{LaiLee23}. That realization applies to the transplanted
problem of this paper without change, since the strong form
$-\Delta_{\sphere}\tilde u = w\tilde f$ involves only the spherical
Laplacian of the Bernstein basis, algebraic in the coefficients by the
formulas of Section~\ref{sec:background}, and the smoothness matrix is
the same; we use the Galerkin form because it keeps the stiffness
matrix symmetric positive semidefinite and free of the weight, which
the eigenproblem needs. Surfaces other than the sphere enter that line
of work through the sphere-like construction of
\cite{AlfNS96a, AlfNS96b}, a radial graph identified with the sphere by
$x \mapsto x/\abs{x}$, which requires $M$ to be star-shaped and puts
the surface into both the stiffness and the mass form through a
general pulled-back metric. Under the conformal parametrization used
here the Dirichlet form is invariant (Lemma~\ref{lem:transplant}), the
stiffness matrix is the round-sphere matrix of \cite{BarLai06}, and
the surface enters only through the scalar weight $w$;
the nonhomogeneous spaces of \cite{BarLaiShum06} would lift the parity
restriction at the cost of Proposition~\ref{prop:parity}.

\emph{Finite elements on surfaces.} The dominant paradigm is the
surface finite element method of Dziuk \cite{Dziuk88}, surveyed in
\cite{DziukElliott13} and, for the parametric, trace and narrow-band
variants with a full account of the geometric error, in Bonito, Demlow
and Nochetto \cite{BonitoDemlowNochetto20}: the surface is approximated
by an interpolating polyhedron or an isoparametric lift, and a
geometric consistency error enters alongside the Galerkin error.
Higher-order parametric elements and their pointwise and a posteriori
estimates are due to Demlow \cite{Demlow09} and Bonito and Demlow
\cite{BonitoDemlow19}, adaptivity on implicit surfaces to Demlow and
Dziuk \cite{DemlowDziuk07}, and the a priori theory for the
Laplace--Beltrami eigenvalue problem to Bonito, Demlow and Owen
\cite{BonitoDemlowOwen18}. In all of these high order requires curved
elements and lifted quadrature, and $C^1$ conformity on surfaces is
delicate. Here, once a conformal parametrization is known, the geometry
is represented through a smooth scalar weight with no further
approximation, the discretization introduces no geometric consistency
error, $C^r$ smoothness comes from the edge functionals, and the
stiffness matrix is independent of the surface. The price is the
conformal factor itself, which for a surface given only as a triangle
mesh must be computed from an approximately conformal map
\cite{GuYau03,KSBC12}, the subject of a companion paper.

\emph{Eigenvalue approximation and spectral geometry.} The Galerkin
approximation theory of symmetric eigenvalue problems is classical
\cite{BabuskaOsborn91, Boffi10}; Hersch's inequality \cite{Hersch70}
and its extension along the spectrum \cite{KNPP21} give a posteriori
bounds that every computed genus-zero spectrum must respect.

\section{Spherical splines in broken Bernstein--B\'ezier form}
\label{sec:background}

This section sets out the spline machinery in the form in which it is
used throughout the paper and in the code. Every spherical triangle
keeps its own Bernstein--B\'ezier coefficients. Smoothness across an
edge is a finite set of linear functionals of the coefficients on the
two triangles that share the edge. The conforming spline space is the
kernel of the matrix obtained by stacking these functionals, and no
global basis is ever constructed. This is the Bernstein-constraint
representation of \cite{Kapita26BCC}, specialized to the sphere; the
ingredients are those of \cite[Ch.~13]{LaiSchumaker07}, and references of
the form Theorem~13.30 are to that book. The section closes with the
structural features of the spherical theory that shape everything that
follows.

\subsection{Spherical triangulations}\label{ssec:triangulations}

A spherical triangle $T=\langle v_1,v_2,v_3\rangle$ is the region of
$\sphere$ bounded by the great-circle arcs joining $v_1,v_2,v_3\in
\sphere$, assumed to lie in an open hemisphere so that the vertices are
linearly independent in $\R^3$. A spherical triangulation $\tri$ is a
finite set of spherical triangles covering $\sphere$, any two meeting
in a vertex, an edge, or not at all, with $N_V$, $N_E$, $N_T$
vertices, edges and triangles. Every edge belongs to exactly two
triangles, and Euler's relation gives
\begin{equation}\label{eq:euler}
  N_E=3N_V-6,\qquad N_T=2N_V-4
\end{equation}
(Theorem~13.33). The mesh size $\abs{\tri}$ is the longest geodesic
edge and $\theta_\tri$ the smallest angle. All experiments use
icosahedral refinements: level $L$ has $N_T=20\cdot4^{L}$ triangles,
obtained by splitting every triangle of level $L-1$ at the geodesic
midpoints of its edges.

\subsection{Spherical barycentric coordinates and Bernstein polynomials}
\label{ssec:construction}

Let $V_T=[\,v_1\ v_2\ v_3\,]\in\R^{3\times3}$ be the matrix whose
columns are the vertices of $T$. It is invertible because $T$ lies in a
hemisphere. The spherical barycentric coordinates of a point $x\in\R^3$
relative to $T$ are
\begin{equation}\label{eq:sphbary}
  b(x)=(b_1,b_2,b_3)=V_T^{-1}x,\qquad\text{so that}\qquad
  x=b_1v_1+b_2v_2+b_3v_3 .
\end{equation}
They are linear, not affine, functions of $x$: nonnegative on $T$, with
$b_1+b_2+b_3>1$ in the interior of $T$. The spherical Bernstein basis
polynomials of degree $d$ relative to $T$ are
\begin{equation}\label{eq:sphbern}
  B^d_{ijk}(x)=\frac{d!}{i!\,j!\,k!}\,b_1(x)^ib_2(x)^jb_3(x)^k,
  \qquad i+j+k=d ,
\end{equation}
restricted to $\sphere$. Each is the restriction of a homogeneous
trivariate polynomial of degree $d$, and the $\binom{d+2}{2}$ functions
\eqref{eq:sphbern} span the space $\mathcal B_d$ of spherical polynomials
of degree $d$, which is exactly the restriction to $\sphere$ of all
homogeneous polynomials of degree $d$ in $\R^3$ (Section~13.1). Since
$\mathcal B_d$ does not depend on $T$, every $p\in\mathcal B_d$ has a
unique B-form relative to every spherical triangle,
\begin{equation}\label{eq:bform}
  p=\sum_{i+j+k=d}c_{ijk}B^d_{ijk},
\end{equation}
and we associate the B-coefficient $c_{ijk}$ with the spherical domain
point
\begin{equation}\label{eq:domainpoints}
  \xi_{ijk}=\frac{iv_1+jv_2+kv_3}{\abs{iv_1+jv_2+kv_3}},\qquad i+j+k=d .
\end{equation}
We write $c_T\in\R^{\binom{d+2}{2}}$ for the vector of B-coefficients
of $p$ on $T$, ordered by domain points. Decomposing into spherical
harmonics,
\begin{equation}\label{eq:parity0}
  \mathcal B_d=\mathcal Y_d\oplus\mathcal Y_{d-2}\oplus\cdots
  \qquad(\text{down to }\mathcal Y_0\text{ or }\mathcal Y_1\text{ by parity}),
\end{equation}
where $\mathcal Y_\ell$ is the $(2\ell+1)$-dimensional space of
spherical harmonics of exact degree $\ell$ (Theorem~13.22).

For a spherical triangulation $\tri$ the broken spline space is
\begin{equation}\label{eq:broken}
  S^{-1}_d(\tri)=\{\,s:\sphere\to\R:\ s|_T\in\mathcal B_d\ \ \forall T\in\tri\,\},
\end{equation}
with no continuity across edges. A function in $S^{-1}_d(\tri)$ is
determined by the concatenation
\begin{equation}\label{eq:brokencoeff}
  c=(c_{T_1},\dots,c_{T_{N_T}})\in\R^N,\qquad N=N_T\binom{d+2}{2},
\end{equation}
of its B-coefficient vectors on the triangles of $\tri$. No coefficient
is shared between triangles, and a domain point on a common edge carries
one coefficient from each of the two triangles that contain it. The
spline spaces of interest are
\begin{equation}\label{eq:Srd}
  S^r_d(\tri)=\{\,s\in C^r(\sphere):\ s|_T\in\mathcal B_d\ \ \forall T\in\tri\,\}
  \subset S^{-1}_d(\tri),\qquad r\ge0 .
\end{equation}
Every planar Bernstein--B\'ezier algorithm whose input is a vector of
barycentric coordinates applies to \eqref{eq:sphbern} verbatim, since
the algorithms never use $b_1+b_2+b_3=1$: de Casteljau evaluation,
directional differentiation, and subdivision
\cite[Ch.~13]{LaiSchumaker07}. Since every global
spherical polynomial restricts to each triangle as a member of
$\mathcal B_d$ and is smooth,
\begin{equation}\label{eq:global}
  \mathcal B_d\subset S^r_d(\tri)\qquad\text{for every }r\ge0 .
\end{equation}

\subsection{Derivatives of the homogeneous extension}\label{ssec:derivatives}

Regard $B^d_{ijk}$ as the homogeneous polynomial \eqref{eq:sphbern} on
$\R^3$. For a direction $z\in\R^3$ let $D_z$ be the directional
derivative and $a=(a_1,a_2,a_3)=V_T^{-1}z$ the spherical barycentric
coordinates of $z$. Then
\begin{equation}\label{eq:bern-deriv}
  D_zB^d_{ijk}=d\bigl(a_1B^{d-1}_{i-1,j,k}+a_2B^{d-1}_{i,j-1,k}+a_3B^{d-1}_{i,j,k-1}\bigr),
\end{equation}
with the convention that a Bernstein polynomial with a negative index
is zero (Theorem~13.28). Taking $z=e_1,e_2,e_3$ gives the ambient
gradient of the homogeneous extension $H_p$ of $p\in\mathcal B_d$ in
coefficient form, with $\alpha=(i,j,k)$ and $e_m$ the $m$th unit
multi-index,
\begin{equation}\label{eq:gradmatrix}
  \frac{\partial H_p}{\partial x_\ell}
  =d\sum_{\abs{\alpha}=d}c_\alpha\sum_{m=1}^3(V_T^{-1})_{m\ell}\,B^{d-1}_{\alpha-e_m},
  \qquad \ell=1,2,3 .
\end{equation}
We write $G_T$ for the resulting sparse matrix, the Bernstein gradient
matrix of \cite{Kapita26BCC} with the rows of $V_T^{-1}$ in the place
of the planar barycentric gradients. Euler's identity
$x\cdot\nabla H_p(x)=d\,H_p(x)$ gives, for $\abs{x}=1$, the tangential
gradient
\begin{equation}\label{eq:tangrad}
  \nabla_{\sphere}p(x)=(I-xx^{\top})\nabla H_p(x)=\nabla H_p(x)-d\,p(x)\,x .
\end{equation}
The first form is what the assembly evaluates at quadrature points; the
second shows that the tangential gradient is algebraic in the
coefficients, with ambient components in $\mathcal B_{d+1}$
(Theorem~13.15).

\subsection{Restriction to an edge}\label{ssec:edge}

Let $e=\langle v_1,v_2\rangle$ be the edge of $T$ opposite $v_3$; it is
an arc of the great circle through $v_1$ and $v_2$. On $e$ the third
coordinate vanishes, $b_3=0$, so by \eqref{eq:sphbern}
\begin{equation}\label{eq:edge-restriction}
  p|_e=\sum_{i+j=d}c_{ij0}\,B^d_{ij0}|_e,\qquad
  B^d_{ij0}(x)=\binom di\,b_1^ib_2^j\ \ \text{for }x=b_1v_1+b_2v_2\in e .
\end{equation}
The $B^d_{ij0}|_e$ are the linearly independent circular Bernstein
polynomials on the arc \cite{AlfNS95}, so two polynomials from the two
triangles sharing $e$ agree on $e$ if and only if their coefficients at
the $d+1$ shared domain points agree. In arc length the restriction is
a trigonometric polynomial (Remark~13.4), which is why quadrature is
unavoidable in assembly.

\subsection{Smoothness functionals across an edge}\label{ssec:smoothness}

Let $T=\langle v_1,v_2,v_3\rangle$ and $\widetilde T=\langle v_4,v_3,v_2\rangle$
share the edge $e=\langle v_2,v_3\rangle$, and let $s\in S^{-1}_d(\tri)$
have B-coefficients $c_{ijk}$ on $T$ and $\widetilde c_{ijk}$ on
$\widetilde T$, each indexed by the vertex order of its own triangle.
Let $(\beta_1,\beta_2,\beta_3)=V_T^{-1}v_4$ be the spherical barycentric
coordinates of the fourth vertex relative to $T$. By Theorem~13.30,
$s$ is $C^r$ across $e$ if and only if
\begin{equation}\label{eq:LS-smoothness}
  \widetilde c_{njk}=\sum_{\nu+\mu+\kappa=n}c_{\nu,\,k+\mu,\,j+\kappa}\,
  B^n_{\nu\mu\kappa}(v_4),\qquad j+k=d-n,\quad n=0,\dots,r,
\end{equation}
where $B^n_{\nu\mu\kappa}(v_4)=\frac{n!}{\nu!\mu!\kappa!}\beta_1^\nu\beta_2^\mu\beta_3^\kappa$.
The first two layers read
\begin{align}
  n=0:&\qquad \widetilde c_{0jk}=c_{0kj},\qquad j+k=d, \label{eq:C0-rows}\\
  n=1:&\qquad \widetilde c_{1jk}=\beta_1c_{1kj}+\beta_2c_{0,k+1,j}+\beta_3c_{0,k,j+1},
  \qquad j+k=d-1 . \label{eq:C1-rows}
\end{align}
The layer $n=0$ is the coefficient matching of Section~\ref{ssec:edge}.
Each $C^1$ condition \eqref{eq:C1-rows} ties one coefficient of
$\widetilde T$ in the first row parallel to $e$ to three coefficients of
$T$, the four circled points of Figure~\ref{fig:stencil} in Appendix~\ref{sm:depcounts}; the stencil is
the planar one and only the numbers $\beta_1,\beta_2,\beta_3$ change.
The layer $n$ involves the coefficients in the rows $0,\dots,n$ parallel
to $e$ on both triangles, and the difference of the two sides of
\eqref{eq:LS-smoothness} is a linear functional of the broken
coefficient vector $c$ supported on the two triangles adjacent to $e$.
We call these the smoothness functionals of order $r$ across $e$ and
write
\begin{equation}\label{eq:Jre}
  J^{(r)}_{0,e}\in\R^{m_e\times N},\qquad m_e=\sum_{n=0}^{r}(d+1-n),
\end{equation}
for the matrix whose rows are these functionals, so that $s$ is $C^r$
across $e$ if and only if $J^{(r)}_{0,e}c=0$. A row of layer $n$ has at
most $\binom{n+2}{2}+1$ nonzero entries. For $r=0$ we write
$C_{0,e}=J^{(0)}_{0,e}$. The subscript $0$ marks the scalar slot of the
constraint complex of \cite{Kapita26BCC}; only that slot is needed in
this paper, since the Laplace--Beltrami problem is scalar and of second
order.

The functionals \eqref{eq:LS-smoothness} describe $C^r$ continuity on
the curved surface exactly: Theorem~13.30 is proved by extending
$s|_T$ and $s|_{\widetilde T}$ homogeneously to the tetrahedra with apex
at the origin, which share the planar face $\langle0,v_2,v_3\rangle$,
and applying the trivariate $C^r$ conditions across that face. No
element map is involved. In the terminology of \cite[Sec.~7]{Kapita26BCC},
where the transfer of $C^r$ functionals to curved patches through
element maps is left open, the sphere is the case in which the
polynomials live on the surface itself, so the only geometric input is
$V_T^{-1}v_4$.

Stacking the blocks $J^{(r)}_{0,e}$ over all $N_E$ edges gives the
smoothness matrix of the space,
\begin{equation}\label{eq:Srd-kernel}
  J^{(r)}_0\in\R^{m\times N},\qquad m=N_E\sum_{n=0}^{r}(d+1-n),\qquad
  S^r_d(\tri)=\ker J^{(r)}_0\subset S^{-1}_d(\tri).
\end{equation}
The implementation stores the edge blocks, so the matrix has $O(N_E)$
rows and nonzeros for fixed $d$ and $r$. Its rows are linearly
dependent, since the conditions around a vertex are related, and this
is allowed throughout; the number of dependencies is counted in
Appendix~\ref{sm:depcounts} and confirmed on the
meshes used.

\subsection{Dimension from the rank}\label{ssec:dimension}

The dimension of the spline space is $N-\rank J^{(r)}_0$ whether or
not the rows are independent, with the rank
delivered by the elimination that builds $Z$; closed-form counts
(Theorem~13.45) are recalled in Appendix~\ref{sm:depcounts} and checked against the computed ranks. For $d<3r+2$ the
dimension can depend on the positions of the vertices
\cite{AlfNS96c}, and the rank is then the only reliable count. The
entire construction lives on the linear structure of $\R^3$: a general
closed surface carries no ambient cone of polynomials, so any extension
is a transfer from the sphere, by the radial projection of
\cite{AlfNS96a} for star-shaped surfaces or by the conformal
parametrization of the next section for every smooth closed genus-zero
surface.

\section{The conformal metric and its instances}\label{sec:extension}

\subsection{The conformal metric on the sphere}\label{ssec:setting}

Let $\lambda \in C^\infty(\sphere)$, $\bar g := e^{2\lambda}\gs$, and
$M := (\sphere, \bar g)$. We write $w := e^{2\lambda}$ for the conformal
factor. The area element, gradient and Laplace--Beltrami operator of
$\bar g$ are
\begin{equation}\label{eq:conformal}
  dA_M = e^{2\lambda}\dA, \qquad
  \nabla_M u = e^{-2\lambda}\nabla u, \qquad
  \Delta_M u = e^{-2\lambda}\Delta_{\sphere} u ,
\end{equation}
the last being the conformal transformation law in dimension two
\cite[Ch.~2]{Jost06}. A conformal diffeomorphism of $\sphere$ pulls
$\bar g$ back to another conformal metric with the same geometry, so
$\lambda$ is determined by the geometry only up to the noncompact
M\"obius group \cite{Jost06}, and an unnormalized representative can
concentrate $w$ arbitrarily. We fix this freedom by Hersch's centering
condition \cite{Hersch70},
\begin{equation}\label{eq:hersch}
  \int_{\sphere} x_i\, w \dA = 0, \qquad i=1,2,3,
\end{equation}
where $x_1,x_2,x_3$ are the ambient coordinate functions restricted to
$\sphere$; a centered representative always exists
\cite{Hersch70,LiYau82}. The remaining additive freedom in $\lambda$ is
a global scaling, which we normalize by
\begin{equation}\label{eq:area}
  \operatorname{Area}(M) \;=\; \int_{\sphere} w \dA \;=\; 4\pi,
\end{equation}
so that computed spectra are comparable with those of the round metric;
the Poisson problem needs neither normalization, and the red blood cell
is run in its own units without \eqref{eq:area}. We write $\Lambda_0 :=
\max_{\sphere} \abs{\lambda}$ and $E := e^{\Lambda_0}$, so that $E^{-2}
\le w \le E^{2}$; constants $C$ are generic and depend at most on the
arguments displayed.

Two forms carry the whole computation. On the round sphere define
\begin{equation}\label{eq:forms}
  a(u,v) := \int_{\sphere} \nabla u \cdot \nabla v \dA,
  \qquad
  b_w(u,v) := \int_{\sphere} w\, u\, v \dA,
  \qquad u,v \in W^{1,2}(\sphere),
\end{equation}
and write $\norm{u}_{\Lqw{2}}^2 := b_w(u,u)$.

\begin{lemma}[Mass and energy of the conformal metric]\label{lem:transplant}
For all $u,v \in W^{1,2}(\sphere)$,
\begin{align}
  \int_M u\,v \dA_M &= b_w(u, v), \label{eq:L2id}\\
  \int_M \nabla_M u \cdot_{\bar g} \nabla_M v \dA_M &= a(u, v),
  \label{eq:dirid}
\end{align}
and $E^{-1}\norm{u}_{L^2(\sphere)} \le \norm{u}_{\Lqw{2}} \le E\,
\norm{u}_{L^2(\sphere)}$. In particular $W^{1,2}(M) = W^{1,2}(\sphere)$
with equivalent norms.
\end{lemma}

\begin{proof}
\eqref{eq:L2id} is the first identity of \eqref{eq:conformal}. For
\eqref{eq:dirid}, $\bar g(\nabla_M u, \nabla_M v) = e^{-2\lambda}\,
\nabla u \cdot \nabla v$ pointwise by the second identity; multiplying
by $dA_M = e^{2\lambda}\dA$ the factors cancel. This cancellation is
the conformal invariance of the Dirichlet energy, special to dimension
two. The norm comparison is $e^{-2\Lambda_0} \le w \le e^{2\Lambda_0}$.
\end{proof}

The approximation power of the spherical spline spaces
is Theorem~14.26: if $\tri$ covers $\sphere$ with
$\abs{\mathrm{star}(v)} < 1$ for every vertex and $\abs{\tri} \le 1/6$,
if $d \ge 3r+2$, and if $0 \le m \le d$ with $m \equiv d \pmod 2$, then
for every $g \in W^{m+1,2}(\sphere)$ there is $s \in S^r_d(\tri)$ with
\begin{equation}\label{eq:jackson}
  \abs{g - s}_{k,2,\sphere} \;\le\; C\bigl(d, \theta_\tri\bigr)\,
  \abs{\tri}^{\,m+1-k}\, \abs{g}_{m+1,2,\sphere},
  \qquad 0 \le k \le m ,
\end{equation}
where $\theta_\tri$ is the smallest angle of $\tri$ and
$\abs{\cdot}_{m+1,2,\sphere}$ is the seminorm of Section~14.5 of
\cite{LaiSchumaker07}. By Lemma~\ref{lem:transplant} the same $s$
approximates $g$ in the norms of $M$: with $m = d$ and $k = 0, 1$,
\begin{equation}\label{eq:approxM}
  \norm{g - s}_{L^2(M)}
  \;+\; \abs{\tri}\; \abs{g - s}_{1,2,M}
  \;\le\;
  C(d,\theta_\tri)\, E\; \abs{\tri}^{\,d+1}\,
  \abs{g}_{d+1,2,\sphere},
\end{equation}
since the energy seminorm of $M$ equals that of the round sphere by
\eqref{eq:dirid} and the $L^2(M)$ norm is the weighted norm by
\eqref{eq:L2id}. The spline spaces themselves need no change: $S^r_d(\tri)$ is a space of functions on $\sphere$, hence
on $M$, and $C^r$ smoothness, the smoothness matrix $J^{(r)}_0$, the
null-space matrix $Z$, the dimension $N - \rank J^{(r)}_0$, determining
sets and local bases are the same for every $\lambda$, because none of
them involves the metric.

\subsection{Instances: how the conformal factor is obtained}
\label{ssec:computew}\label{ssec:pullback}

\emph{An embedded surface with a conformal parametrization.} Let
$\Sigma \subset \R^3$ be a smooth closed surface of genus zero with its
induced metric, and let $\Phi\colon\sphere\to\Sigma$ be a conformal
diffeomorphism, which exists by uniformization \cite{Jost06}. Then
$\Phi^* g_\Sigma = e^{2\lambda}\gs$ for one $\lambda$, so $\Phi$ is an
isometry from $M$ onto $\Sigma$, and every Laplace--Beltrami problem on
$\Sigma$ is the corresponding problem on $M$ with data $f\circ\Phi$
and solution $u\circ\Phi^{-1}$. The factor is the area magnification
of $\Phi$: at $x\in\sphere$ with orthonormal tangent vectors
$\tau_1,\tau_2$,
\begin{equation}\label{eq:w-general}
  w(x) \;=\; \abs{\,d\Phi_x(\tau_1)\times d\Phi_x(\tau_2)\,}
  \;=\; \sigma_1(x)\,\sigma_2(x),
\end{equation}
where $\sigma_1\ge\sigma_2>0$ are the singular values of $d\Phi_x$, and
conformality is $\sigma_1=\sigma_2$. In colatitude and longitude
$(\Theta,\phi)$,
\begin{equation}\label{eq:w-coords}
  w(\Theta,\phi) \;=\;
  \frac{\abs{\,\partial_\Theta\Phi\times\partial_\phi\Phi\,}}{\sin\Theta},
\end{equation}
which a code evaluates when $\Phi$ is given with its derivatives; with
$\Phi$ alone, centered differences with step $10^{-5}$ reproduced the
geometry diagnostics of the experiments below to about ten digits, and
the ratio $\sigma_1/\sigma_2$ from the same differences measures the
deviation from conformality, assumed here negligible against the
discretization error (an approximately conformal map perturbs the
metric in both forms; its treatment is deferred). The values of $w$
are needed only at quadrature points.

\emph{A surface of revolution.} Let $\Sigma$ be given by a generating
curve $(r(\eta), z(\eta))$, $\eta\in[0,\pi]$, rotated about the
$z$-axis, with $r>0$ on $(0,\pi)$ and the curve meeting the axis
regularly at both ends; the spheroid, the dumbbell and the red blood
cell are of this form. The metric of $\Sigma$ in $(\eta,\phi)$ is
$(r'^2+z'^2)\,d\eta^2 + r^2\,d\phi^2$, and its Mercator variable is
\begin{equation}\label{eq:merc-gen}
  t(\eta) = \ln\tan(\eta/2) + G(\eta), \qquad
  G'(\eta) = \frac{\sqrt{r'^2 + z'^2}}{r} - \frac{1}{\sin\eta},
\end{equation}
where $G'$ extends smoothly to the poles. The round sphere has the same
structure with $G\equiv0$, and matching the two Mercator variables at
the equator defines the conformal map $\Theta\mapsto\eta(\Theta)$
through the monotone scalar equation
\begin{equation}\label{eq:merc-match}
  \ln\tan(\eta/2) + \widetilde G(\eta) = \ln\tan(\Theta/2), \qquad
  \widetilde G = G - G(\pi/2),
\end{equation}
solved by a safeguarded Newton iteration, and $\Phi(\Theta,\phi) =
(r(\eta)\cos\phi,\, r(\eta)\sin\phi,\, z(\eta))$. The factor is
$w = r(\eta)^2/\sin^2\Theta$, which we evaluate in the numerically
stable half-angle form
\begin{equation}\label{eq:w-stable}
  w = \Bigl(\frac{r(\eta)}{\sin\eta}\Bigr)^{2}\, e^{-2\widetilde G(\eta)}
  \Bigl( \frac{1 + T^2}{1 + \tau^2} \Bigr)^{\!2},
  \qquad \tau = \tan(\eta/2), \quad T = \tan(\Theta/2)
  = \tau\, e^{\widetilde G(\eta)},
\end{equation}
which is accurate at the poles. Star-shapedness is not used, so
profiles with concavities are admissible. Equatorial symmetry makes the
Hersch centering \eqref{eq:hersch} automatic, and the exact area
$2\pi\int_0^\pi r\sqrt{r'^2+z'^2}\,d\eta$ is an independent identity
for the computed factor, against which $\int_{\sphere} w\dA$ under the
assembly quadrature is compared as the first check of every run; in the
experiments the agreement is at roundoff. For a radial graph
$\{R(\psi)\,v(\psi,\phi)\}$ over the sphere the construction is the
special case $r = R\sin\psi$, $z = R\cos\psi$, with $G' =
(\sqrt{R^2+R'^2} - R)/(R\sin\psi)$ and prefactor $R(\eta)^2$ in
\eqref{eq:w-stable}; the spheroid is the case with closed-form $G$.

\emph{A prescribed factor.} Any smooth $\lambda$ defines an instance;
the computation needs only $w$ at the quadrature points, and
Section~\ref{sec:numerics} includes one.

\emph{A surface given as a triangle mesh.} The map $\Phi$ must then be
computed, by harmonic energy minimization with M\"obius normalization
\cite{GuYau03} or conformalized mean curvature flow \cite{KSBC12}, and
is only approximately conformal; that instance is the subject of a
companion paper.

\section{The Poisson equation as a weighted problem on the round sphere}
\label{sec:poisson}

The weak form of \eqref{eq:poissonM} is written in the round metric
through the two identities of Lemma~\ref{lem:transplant} and then in
the broken coefficients of Section~\ref{sec:background}, with
conformity imposed through the smoothness matrix. The section closes
with the eigenproblem, which uses the same matrices.

\subsection{The weak form in the round metric}\label{ssec:weakform}

Let $f \in L^2(M)$ with $\int_M f \dA_M = 0$. The weak form of
\eqref{eq:poissonM} is: find $u \in W^{1,2}(M)$ with
$\int_M u \dA_M = 0$ such that
\begin{equation}\label{eq:weakM}
  \int_M \nabla_M u \cdot_{\bar g} \nabla_M v \dA_M \;=\; \int_M f\, v \dA_M
  \qquad \text{for all } v \in W^{1,2}(M).
\end{equation}
By \eqref{eq:dirid} on the left and \eqref{eq:L2id} on the right, and
since $W^{1,2}(M) = W^{1,2}(\sphere)$, this is the problem on the round
sphere: find $u \in W^{1,2}(\sphere)$ with $b_w(u, 1) = 0$ such that
\begin{equation}\label{eq:poissonS}
  a(u, v) \;=\; b_w(f, v)
  \qquad \text{for all } v \in W^{1,2}(\sphere),
\end{equation}
with the forms $a$ and $b_w$ of \eqref{eq:forms}; the compatibility
condition is $b_w(f, 1) = 0$. Three facts organize what follows. The
metric enters only through $w$, on the right-hand side and in the
constraint; $a$ is the Dirichlet form of the round sphere.
The problem is well posed: on $\{v : b_w(v,1) = 0\}$ the form $a$ is
coercive, since $a(v,v) \ge \mu_1(M)\, b_w(v,v)$ there by
Proposition~\ref{prop:exact} below with $\mu_1(M) > 0$, so Lax--Milgram
gives a unique solution, which satisfies \eqref{eq:poissonS} for all
$v$ because both sides vanish at $v = 1$. The strong form is $-\Delta_{\sphere} u = w f$ by
\eqref{eq:conformal}; for a manufactured $u$ the product $w f$ in the
load vector is $-\Delta_{\sphere} u$ and contains no $w$, which the
first experiment of Section~\ref{sec:numerics} uses to separate the
effect of the factor in the load vector from its effect in the
constraint and the norms.

\subsection{Standing assumptions}\label{ssec:splines}

Let $\tri$ be a spherical triangulation with mesh size $\abs{\tri}$
and smallest angle $\theta_\tri$, and $S^r_d(\tri)=\ker J^{(r)}_0$ the
spline space of Section~\ref{ssec:smoothness}. The standing assumptions
are (A1) $\tri$ covers $\sphere$, $\abs{\mathrm{star}(v)} < 1$ for
every vertex and $\abs{\tri} \le 1/6$, the hypotheses under which the
dimension and approximation theory of spherical splines are available
(Theorems~13.43, 13.45, 14.25, 14.26); (A2) $d$ is even and $d \ge
3r+2$ with $r \ge 0$; (A3) $w = e^{2\lambda}$ with $\lambda$ smooth,
normalized by \eqref{eq:hersch} and, except for surfaces run in
physical units, by \eqref{eq:area}. The parity requirement in (A2) is
not technical: $\mathcal B_d$ contains the constants if and only if
$d$ is even, and by Theorem~13.22
\begin{equation}\label{eq:parity}
  \mathcal B_d = \mathcal Y_0 \oplus \mathcal Y_2 \oplus \cdots \oplus
  \mathcal Y_d \qquad (d \text{ even}).
\end{equation}
Even $d$ puts the constant function, which fixes the mean-value
constraint and is the zero-eigenvalue eigenfunction, in the discrete
space; for odd $d$ it is not representable even on a single triangle
(Theorem~13.23). The inequality $d \ge 3r+2$ is needed only for
\eqref{eq:jackson}.

\subsection{The discrete problem in broken coefficients}\label{ssec:pencil}

Set $S_h := S^r_d(\tri) \subset W^{1,2}(\sphere)$; $r = 0$ is already
conforming, and higher $r$ buys a smoother solution at fewer degrees of
freedom. The discrete problem reads: find $u_h \in S_h$ with
$b_w(u_h, 1) = 0$ such that
\begin{equation}\label{eq:poissonh}
  a(u_h, v_h) = b_w(f, v_h)
  \qquad \text{for all } v_h \in S_h .
\end{equation}

We write it in the broken coefficients of Section~\ref{sec:background}.
On each triangle $T$ the element stiffness matrix $K_T$, the element
weighted mass matrix $M_{w,T}$, and the element load vector $F_T$,
indexed by the multi-indices $\alpha,\beta$ of degree $d$, are
\begin{equation}\label{eq:elementmatrices}
\begin{aligned}
  (K_T)_{\alpha\beta}&=\int_T\nabla_{\sphere}B^d_\alpha\cdot\nabla_{\sphere}B^d_\beta\dA,
  &\qquad
  (M_{w,T})_{\alpha\beta}&=\int_Tw\,B^d_\alpha B^d_\beta\dA,\\
  (F_T)_\alpha&=\int_T w\,f\,B^d_\alpha\dA,
\end{aligned}
\end{equation}
with the tangential gradients given by \eqref{eq:tangrad}. The broken
stiffness matrix $K=\operatorname{diag}_TK_T$ and the broken weighted
mass matrix $M_w=\operatorname{diag}_TM_{w,T}$ are block diagonal of
size $N\times N$, one block per triangle, the broken load vector
$F \in \R^N$ is the concatenation of the $F_T$, and for $c,y\in\R^N$
representing $s,v\in S^{-1}_d(\tri)$ one has $y^{\top}Kc=a(s,v)$,
$y^{\top}M_wc=b_w(s,v)$ and $y^\top F = b_w(f, v)$ with the
forms extended triangle by triangle. Let $c_1 \in \R^N$ be the broken
coefficient vector of the constant function $1$. Since $1 \in
\mathcal B_d$ for even $d$ (Theorem~13.23), its B-form on each triangle
is unique, and its entries are not all equal because $b_1+b_2+b_3$ is
not $1$ on the sphere; $c_1$ is computed once per triangulation and
degree by solving the square Bernstein evaluation system at the domain
points of each triangle, which is exact up to roundoff. Writing $J:=J^{(r)}_0$ for the
smoothness matrix, \eqref{eq:poissonh} is the constrained linear
problem
\begin{equation}\label{eq:constrained-poisson}
  c\in\ker J,\qquad c_1^{\top} M_w c = 0, \qquad
  y^{\top}Kc = y^{\top}F \quad\text{for all }y\in\ker J .
\end{equation}
The broken stiffness matrix $K$ is the round-sphere Dirichlet matrix,
assembled once; $M$ enters only through $w$ in $F$ and $M_w c_1$. This
is the constrained Galerkin problem of \cite{Kapita26BCC} and the
constrained B-form setting of \cite{AwanouLaiWenston06}: the broken
matrices are never modified, and conformity is a side condition on the
coefficient vector; for $w \equiv 1$ it is the discretization of
\cite{BarLai06}. It can be solved on the same space by a null-space
matrix, by the augmented Lagrangian iteration of
\cite{AwanouLaiWenston06}, or as a saddle-point system with Lagrange
multipliers \cite[Sec.~6]{Kapita26BCC}. The experiments use the
null-space method \cite[Sec.~6]{BenziGolubLiesen05}, checked against
the saddle-point form in Section~\ref{ssec:checks}.

\paragraph{The null-space method}
Let $Z\in\R^{N\times n_h}$ have full column rank and $\range Z=\ker J$,
so that $n_h=\dim S_h$ and every conforming coefficient vector is
$c=Z\hat c$ for a unique $\hat c\in\R^{n_h}$. Substituting into
\eqref{eq:constrained-poisson} and appending the mean-value constraint
with a scalar multiplier $\nu$ gives the reduced system
\begin{equation}\label{eq:reducedsystem}
  \begin{bmatrix} Z^{\top} K Z & g \\ g^{\top} & 0 \end{bmatrix}
  \begin{bmatrix} \hat c \\ \nu \end{bmatrix}
  \;=\;
  \begin{bmatrix} Z^{\top} F \\ 0 \end{bmatrix},
  \qquad g := Z^{\top} M_w c_1, \qquad c = Z\hat c .
\end{equation}
The reduced stiffness matrix $Z^{\top}KZ$ is symmetric positive
semidefinite with kernel spanned by the reduced coefficients $\hat c_1$
of the constant; since $g^\top \hat c_1 = \int_{\sphere} w \dA > 0$,
the bordered matrix is nonsingular. The multiplier $\nu$ vanishes
whenever the discrete compatibility $c_1^\top F = 0$ holds to roundoff,
and its size checks the quadrature of the right-hand side. In the
language of \cite[Ch.~5]{LaiSchumaker07}, the columns of $Z$ extend the
coefficients on a determining set to all $N$ broken coefficients.

The construction of $Z$ by sparse rank-revealing elimination of $J$
in a locality order, the saddle-point realization of the same
equations, and the reading of residuals in the broken representation
(the raw residual $Kc-F$ is a Lagrange multiplier in $\range J^\top$
and vanishes only after projection by $Z^\top$) are given in
Appendix~\ref{sm:nullspace}.

\subsection{What error to expect}\label{ssec:expected}

The rates observed in Section~\ref{sec:numerics} follow from the
approximation estimate \eqref{eq:approxM} by standard arguments, which
we state with their hypotheses.

\begin{proposition}\label{prop:rates}
Let $w \in C^\infty(\sphere)$ with $0 < w_{\min} \le w \le w_{\max}$,
let (A1) and (A2) hold, and let $u \in W^{d+1,2}(\sphere)$ solve
\eqref{eq:poissonM}. Then the solution $u_h$ of
\eqref{eq:poissonh} satisfies
\begin{equation}\label{eq:H1rate}
  \abs{u - u_h}_{1,2,M} \le C\, \abs{\tri}^{\,d}\,
  \abs{u}_{d+1,2,\sphere},
  \qquad
  \norm{u - u_h}_{L^2(M)} \le C\, \abs{\tri}^{\,d+1}\,
  \abs{u}_{d+1,2,\sphere},
\end{equation}
with $C$ depending on $d$, $\theta_\tri$, $w_{\min}$, $w_{\max}$ and
$\norm{w}_{C^1(\sphere)}$. For the eigenproblem \eqref{eq:evph}, if
$\mu$ is an eigenvalue of \eqref{eq:evpM} with eigenspace
$E \subset W^{d+1,2}(\sphere)$, then the discrete
eigenvalues converging to $\mu$ satisfy $0 \le \mu_h - \mu \le C
\abs{\tri}^{\,2d}$.
\end{proposition}

\begin{proof}
The energy estimate is C\'ea's lemma \cite[Ch.~2]{BrennerScott08} on
the $w$-mean-zero subspace, where $a$ is coercive by
Proposition~\ref{prop:exact}, with \eqref{eq:approxM} and
\eqref{eq:dirid}. For the $L^2$ estimate let $e = u - u_h$ made
$w$-mean-free and let $z$ solve $-\Delta_{\sphere} z = w e$,
$b_w(z,1)=0$; elliptic regularity on the closed manifold $\sphere$
\cite[Ch.~5]{Taylor11} gives $\norm{z}_{W^{2,2}} \le C w_{\max}
\norm{e}_{L^2}$. Galerkin orthogonality gives $b_w(e,e) = a(z - z_h,
e)$ for every $z_h \in S_h$, \eqref{eq:jackson} with $m = 1$ bounds
$\inf_{z_h}\abs{z - z_h}_{1,2}$ by $C\abs{\tri}\norm{z}_{W^{2,2}}$,
and dividing by $\norm{e}_{\Lqw{2}} \ge w_{\min}^{1/2}\norm{e}_{L^2}$
and using the energy estimate gives the $L^2$ rate in the weighted
norm, which is the $L^2(M)$ norm by \eqref{eq:L2id}. The eigenvalue statement is the
Babu\v{s}ka--Osborn estimate for conforming approximation of a
symmetric pencil with compact solution operator
\cite[Sec.~7]{BabuskaOsborn91}: the eigenvalue error is bounded by the
square of the best energy approximation of the eigenspace, which
\eqref{eq:approxM} bounds by $\abs{\tri}^{\,d}$.
\end{proof}

The growth of the constants with the conformal factor and with the
position along the spectrum is the subject of a separate paper.

\subsection{The Laplace--Beltrami eigenproblem}
\label{ssec:eigen}

The eigenvalue problem on $M$ is: find $u \not\equiv 0$ and $\mu \in
\R$ with
\begin{equation}\label{eq:evpM}
  -\Delta_M u = \mu\, u \quad \text{on } M,
\end{equation}
understood weakly in $W^{1,2}(M)$, and on the round sphere the
generalized eigenvalue problem for the pencil $(a,b_w)$: find
$u \not\equiv 0$ and $\mu \in \R$ with
\begin{equation}\label{eq:evpS}
  a(u, v) = \mu\, b_w(u, v)
  \qquad \text{for all } v \in W^{1,2}(\sphere).
\end{equation}

\begin{proposition}[The spectrum as a weighted problem]\label{prop:exact}
$(u,\mu)$ is an eigenpair of \eqref{eq:evpM} if and only if it is an
eigenpair of \eqref{eq:evpS}, and no constant in this correspondence
depends on $\lambda$.
\end{proposition}

\begin{proof}
Insert \eqref{eq:L2id} and \eqref{eq:dirid} into the weak form of
\eqref{eq:evpM}.
\end{proof}

The pencil is symmetric with $b_w$ positive definite, so the spectrum
is real, discrete, $0 = \mu_0 < \mu_1 \le \mu_2 \le \cdots \to \infty$
\cite[Ch.~I]{Chavel84}, with $b_w$-orthogonal eigenfunctions and the
constant eigenfunction at $\mu_0 = 0$. The discrete eigenvalue problem
reads: find $u_h \in S_h \setminus \{0\}$ and $\mu_h \in \R$ with
\begin{equation}\label{eq:evph}
  a(u_h, v_h) = \mu_h\, b_w(u_h, v_h)
  \qquad \text{for all } v_h \in S_h ,
\end{equation}
which in broken coefficients is the constrained pencil
$y^\top K c = \mu_h\, y^\top M_w c$ for all $y \in \ker J$, and with the
null-space matrix the reduced pencil
\begin{equation}\label{eq:matrixpencil}
  \bigl(Z^{\top} K Z\bigr)\, \hat c \;=\; \mu_h\, \bigl(Z^{\top} M_w Z\bigr)\,
  \hat c , \qquad c = Z \hat c ,
\end{equation}
in which $Z^{\top}M_wZ$ is positive definite. The saddle-point form is
the pencil
\begin{equation}\label{eq:saddlepencil}
  \mathcal A x = \mu\, \mathcal M x, \qquad
  \mathcal A=\begin{bmatrix}K&J_b^{\top}\\J_b&0\end{bmatrix},\qquad
  \mathcal M=\begin{bmatrix}M_w&0\\0&0\end{bmatrix};
\end{equation}
its finite eigenvalues with nonzero
primal component are exactly those of \eqref{eq:evph}, and the
multiplier directions contribute only infinite eigenvalues
\cite[Sec.~6]{Kapita26BCC}, so shift-invert Arnoldi iteration
\cite{LSY98} on $(\mathcal A - \sigma\mathcal M)^{-1}\mathcal M$ finds
the constrained spectrum without any null-space matrix. Three
facts follow from the structure. First, since $S_h \subset W^{1,2}(\sphere)$, the Courant--Fischer
principle \cite[Ch.~I]{Chavel84} applied over subspaces of $S_h$ gives
the one-sided bounds
\begin{equation}\label{eq:onesided}
  \mu_{i,h} \;\ge\; \mu_i , \qquad i = 0, 1, \ldots, n_h - 1 .
\end{equation}
Second, since $d$ is even, the constants lie in $S_h$ by
\eqref{eq:global}, and $a(1, v_h) = 0 = \mu\, b_w(1, v_h)$ shows that
$(1, 0)$ is an exact discrete eigenpair, $\mu_{0,h} = 0 = \mu_0$. Third,
the same principle applied to a perturbed weight gives a spectrally
uniform stability estimate: if $\widehat w$ is the weight actually used
and $\delta := \norm{\widehat w/w - 1}_{L^\infty(\sphere)} < 1$, then
$(1-\delta)b_w \le b_{\widehat w} \le (1+\delta)b_w$ as quadratic
forms, and the monotonicity of the Courant--Fischer minimum in the mass
form gives
\begin{equation}\label{eq:weightpert}
  \frac{\mu_i}{1+\delta} \;\le\; \widehat\mu_i \;\le\; \frac{\mu_i}{1-\delta},
  \qquad i \ge 0,
\end{equation}
for the eigenvalues $\widehat\mu_i$ of $(a, b_{\widehat w})$; the
same argument covers quadrature errors in $K$ and $M_w$. Errors in the
weight and the quadrature therefore enter the spectrum as relative
errors uniform along the spectrum. Convergence at the rate $\abs{\tri}^{2d}$ is
Proposition~\ref{prop:rates}, with no geometric consistency term, in
contrast with the surface finite element eigenvalue theory of
\cite{BonitoDemlowOwen18}. One structural statement is specific to the
sphere and serves as an implementation test.

\begin{proposition}[Parity fingerprint]\label{prop:parity}
Let $\lambda \equiv 0$, so that $w \equiv 1$ and \eqref{eq:evph} is the
discrete eigenvalue problem of the round metric, and let $d$ be even.
Then for every even $\ell$ with $0 \le \ell \le d$, every spherical
harmonic $Y \in \mathcal Y_\ell$ is an exact discrete eigenfunction with
exact eigenvalue $\ell(\ell+1)$:
\begin{equation*}
  a(Y, v_h) = \ell(\ell+1)\, b_1(Y, v_h)
  \qquad \text{for all } v_h \in S_h ,
\end{equation*}
so the discrete spectrum contains $\ell(\ell+1)$ with multiplicity at
least $2\ell + 1$, and for $\abs{\tri}$ small enough the discrete
eigenvalue cluster converging to $\ell(\ell+1)$ has total multiplicity
exactly $2\ell+1$. The eigenvalues of odd order $\ell$ converge from
above at the rate $\abs{\tri}^{2d}$.
\end{proposition}

\begin{proof}
By \eqref{eq:parity} and \eqref{eq:global}, $\mathcal Y_\ell \subset
\mathcal B_d \subset S_h$ for even $\ell \le d$, and $-\Delta_{\sphere}
Y = \ell(\ell+1) Y$ for $Y \in \mathcal Y_\ell$
\cite[Ch.~2]{AtkinsonHan12}, so $\bigl(Y, \ell(\ell+1)\bigr)$ satisfies
\eqref{eq:evph} verbatim; the $2\ell+1$ harmonics of degree $\ell$ are
independent in $S_h$, which gives the multiplicity bound. By
\cite[Sec.~7]{BabuskaOsborn91}, for $\abs{\tri}$ small enough
exactly $2\ell+1$ discrete eigenvalues, counted with multiplicity, lie
in a neighborhood of $\ell(\ell+1)$ that excludes the other eigenvalues
of the sphere, and these are the cluster.
\end{proof}

At $d = 6$ the discrete spectrum of the round metric therefore
contains $0$, $6$, $20$ and $42$ with multiplicities $1$, $5$, $9$ and
$13$ up to quadrature error, while the clusters at $2$, $12$ and $30$
converge at twelfth order;
a quadrature defect perturbs the even clusters, a constraint-rank
defect changes the multiplicities, and an assembly error destroys the
rate. For every $\lambda$, Hersch's inequality
$\mu_1(M)\operatorname{Area}(M) \le 8\pi$ \cite{Hersch70} and its
extension $\mu_k(M)\operatorname{Area}(M) \le 8\pi k$ \cite{KNPP21}
bound the eigenvalues from above; since the discrete eigenvalues lie
above the exact ones by \eqref{eq:onesided}, a normalized discrete
eigenvalue exceeding $2k$ signals underresolution or a failure of the
normalization \eqref{eq:hersch}.

\section{The algorithm}\label{sec:algorithm}

Algorithm~\ref{alg:poisson} lists the computation in execution order.
The conformal factor is obtained and checked before any equation is
discretized, and the smoothness, null-space and stiffness matrices are
built once per triangulation and degree and reused for every
$\lambda$.

\begin{algorithm}[htbp]
\caption{Poisson equation for a conformal metric on the sphere by spherical splines}
\label{alg:poisson}
\begin{algorithmic}[1]
\REQUIRE the conformal factor $w = e^{2\lambda}$, as a formula, or as
a conformal parametrization $\Phi$ or a generating curve from which it
is computed; right-hand side $f$; spherical triangulation $\tri$; degree
$d$ (even) and smoothness $r$
\ENSURE broken coefficient vector $c$ of the spline solution $u_h\approx u$
\STATE \textbf{Geometry.} Build the routine $x\mapsto w(x)$: for a
surface of revolution solve \eqref{eq:merc-match} for $\eta(\Theta)$
and evaluate \eqref{eq:w-stable}; for a parametrized surface evaluate
\eqref{eq:w-coords}; for a prescribed $\lambda$ evaluate $e^{2\lambda}$.
Check $\int_{\sphere} w\dA$ against the exact area when one is known
and the Hersch moments \eqref{eq:hersch} against zero.
\STATE \textbf{Spline space (once per $\tri,d,r$).} Assemble the edge
blocks \eqref{eq:LS-smoothness} into $J$; eliminate to reduced row
echelon form in the locality order and form $Z$ by \eqref{eq:rref};
compare $N-\rank J$ with the closed-form counts of Appendix~\ref{sm:depcounts}.
\STATE \textbf{Stiffness (once per $\tri,d$).} Assemble the broken
$K=\operatorname{diag}_T K_T$ from \eqref{eq:elementmatrices} with the
tangential gradient \eqref{eq:tangrad}; form and store $Z^\top K Z$.
\STATE \textbf{Weighted data.} At the quadrature points of each
triangle evaluate $w$, $f$ (for an embedded surface, $f\circ\Phi$) and
the Bernstein basis;
assemble the load vector $F$ and the constraint vector $M_w c_1$ of
\eqref{eq:elementmatrices}, where $c_1$ interpolates the constant.
Check $c_1^\top F$ against zero.
\STATE \textbf{Solve.} Solve the bordered system \eqref{eq:reducedsystem}
by a sparse direct solver; set $c = Z\hat c$. Check the multiplier
$\nu$ against zero and the projected residual $Z^\top(Kc - F)$ against
zero.
\STATE \textbf{Output.} Evaluate $u_h$ from $c$ by de Casteljau on any
triangle; for an embedded surface the value at $p\in\Sigma$ is
$u_h(\Phi^{-1}(p))$. Errors are measured in $L^2(M)$ and $H^1(M)$
through \eqref{eq:L2id} and \eqref{eq:dirid}.
\end{algorithmic}
\end{algorithm}

The formulas each step evaluates, the Newton iteration and half-angle
form for the conformal factor of a surface of revolution, the radial
projection of the reference quadrature rule, the quadrature formulas
for the load and constraint vectors, the solver diagnostics, and the
cost of a new surface or a new right-hand side are given in
Appendix~\ref{sm:steps}. $J$, $Z$ and $Z^\top K Z$ are computed once per $(\tri,d,r)$; a new
$\lambda$ costs one pass over the quadrature points and one
factorization of \eqref{eq:reducedsystem}; a new right-hand side costs one pass and
one triangular solve; the eigenproblem costs one assembly of $M_w$ and
one Lanczos run.

\section{Numerical experiments}\label{sec:numerics}

The smoothness matrix $J$, the null-space matrix $Z$ and the stiffness
matrix $K$ are reused unchanged across instances; the metric is supplied
through the conformal factor of step~1 of Algorithm~\ref{alg:poisson}. Local matrices use a conical Gauss product rule of order $12$ per
direction on each spherical triangle. The bordered system \eqref{eq:reducedsystem} is solved by a
sparse direct solver, and the reduced pencil \eqref{eq:matrixpencil} by
shift-inverted Lanczos at $\sigma=-1$. Icosahedral level $\ell$ has $20\cdot4^{\ell}$ triangles and geodesic
mesh sizes $0.628$, $0.326$, $0.165$, $0.083$ at levels $1$ to $4$; the
degrees are $(d,r) = (2,0)$, $(4,0)$ and $(6,1)$, the last the smallest
even degree admitted by $d \ge 3r+2$ for $C^1$ splines.

The instances are the round metric, $\lambda \equiv 0$; three surfaces
of revolution, the prolate spheroid $x^2+y^2+(z/c)^2 = 1$ with $c =
3/2$, the dumbbell $R(\psi) = 1 - \tfrac34\sin^2\psi$ with two bulbs of
radius one joined by a neck of radius $1/4$, and the Evans--Fung red
blood cell \cite{EvansFung72}, a biconcave discocyte run in physical
units (micrometers), whose factors are computed from the generating
curves by \eqref{eq:merc-gen}, with the spheroid's Mercator integral in
closed form \eqref{eq:mercator}; and a prescribed factor with no
embedding, $\lambda = 0.6\,(x_1^2 - x_2^2) - 0.4\,x_3^2$, whose
symmetry satisfies \eqref{eq:hersch}. The three surfaces of revolution
pass the area identity to relative error below $3\times10^{-15}$;
$\osc\lambda$ is $0$, $0.535$, $2.23$, $0.373$ and $1.2$ respectively,
and the dumbbell is the hard case, where $e^{2\osc\lambda} \approx 86$
enters the constants of \eqref{eq:approxM} (Table~\ref{tab:geomdiag}).

\subsection{Algebraic checks}\label{ssec:checks}

The mesh, quadrature, smoothness matrix, null-space matrix and solvers
are checked independently of the metric: Euler relations, quadrature
area, rank and dependency count of $J$, reproduction of the constant
and of $xy$, agreement of the null-space and saddle-point realizations
to $6\times10^{-14}$, one-sidedness \eqref{eq:onesided} to
$3\times10^{-13}$, and product-to-moment assembly \cite{AinsworthAD11}
against direct quadrature to $10^{-15}$ (Appendix~\ref{sm:checks}).

\subsection{The Poisson equation}\label{ssec:num-poisson}

Two tests separate the two places where the metric enters. In Test 1
the exact solution is $u = e^{x_1}x_2 + x_1x_3^3$, not a spherical
polynomial of any degree, with $f = -\Delta_M u =
-w^{-1}\Delta_{\sphere}u$ computed symbolically; the product $wf$ in
the load vector is then $-\Delta_{\sphere}u$, so the discrete solutions
for the different instances differ only through the constraint vector
$M_w c_1$ and the weighted norm in which the error is measured, and
Test 1 isolates the spline approximation from the metric. In Test 2
the right-hand side is prescribed on the embedded surface, $f(p) =
p_1$ made mean-free, so the load vector sees the metric at every
quadrature point; its reference is a dense spherical-harmonic Galerkin
solve of \eqref{eq:poissonS}. Errors are measured through \eqref{eq:L2id} and \eqref{eq:dirid}
with the assembly quadrature. Table~\ref{tab:poisson1} reports
Test 1 and Table~\ref{tab:poisson2main} the degree-$4$ and $6$ rows of
Test 2, whose full record is in Appendix~\ref{sm:test2}.

\begin{table}[htbp]
\centering
\caption{Poisson problem, Test 1: exact solution $u = \tilde u \circ \Phi^{-1}$ with $\tilde u = e^{x_1} x_2 + x_1 x_3^3$ on the sphere, right-hand side $f = -\Delta_M u$, icosahedral levels $1$ to $4$. The $H^1(M)$ seminorm error is the same on all four surfaces, because the load vector $\int_{\sphere} w \tilde f B_\alpha \dA = -\int_{\sphere} (\Delta_{\sphere} \tilde u) B_\alpha \dA$ does not contain $w$; the $L^2(M)$ error sees $w$ through the constant fixed by the weighted mean and through the weighted norm, and is listed per surface. Rates are against the geodesic mesh size; the $L^2(M)$ rates on the spheroid, dumbbell and red blood cell differ from those of the sphere by at most $0.16$ at level $2$ and by less than $0.03$ at levels $3$ and $4$. The red blood cell is in micrometers.}
\label{tab:poisson1}
\scriptsize
\setlength{\tabcolsep}{1.5pt}
\begin{tabular}{@{}llr ll ll lll@{}}
\toprule
& & & \multicolumn{2}{c}{$H^1(M)$, all surfaces} & \multicolumn{2}{c}{$L^2(M)$, sphere} & \multicolumn{3}{c}{$L^2(M)$, other surfaces} \\
\cmidrule(lr){4-5}\cmidrule(lr){6-7}\cmidrule(lr){8-10}
$(d,r)$ & lev. & $\dim S_h$ & error & rate & error & rate & spheroid & dumbbell & red blood cell \\
\midrule
$(2,0)$ & 2 & 642 & $4.59\,10^{-2}$ & 2.04 & $1.75\,10^{-3}$ & 3.04 & $2.09\,10^{-3}$ & $1.33\,10^{-3}$ & $5.56\,10^{-3}$ \\
 & 3 & 2562 & $1.16\,10^{-2}$ & 2.02 & $2.24\,10^{-4}$ & 3.02 & $2.69\,10^{-4}$ & $1.72\,10^{-4}$ & $7.11\,10^{-4}$ \\
 & 4 & 10242 & $2.92\,10^{-3}$ & 2.01 & $2.82\,10^{-5}$ & 3.02 & $3.39\,10^{-5}$ & $2.18\,10^{-5}$ & $8.96\,10^{-5}$ \\
\midrule
$(4,0)$ & 2 & 2562 & $1.70\,10^{-4}$ & 4.19 & $3.70\,10^{-6}$ & 5.20 & $4.46\,10^{-6}$ & $2.77\,10^{-6}$ & $1.17\,10^{-5}$ \\
 & 3 & 10242 & $1.07\,10^{-5}$ & 4.06 & $1.17\,10^{-7}$ & 5.07 & $1.42\,10^{-7}$ & $8.84\,10^{-8}$ & $3.71\,10^{-7}$ \\
 & 4 & 40962 & $6.70\,10^{-7}$ & 4.03 & $3.69\,10^{-9}$ & 5.04 & $4.45\,10^{-9}$ & $2.78\,10^{-9}$ & $1.16\,10^{-8}$ \\
\midrule
$(6,1)$ & 2 & 3206 & $1.40\,10^{-6}$ & 6.04 & $2.57\,10^{-8}$ & 6.92 & $3.25\,10^{-8}$ & $2.38\,10^{-8}$ & $7.98\,10^{-8}$ \\
 & 3 & 12806 & $2.49\,10^{-8}$ & 5.92 & $2.48\,10^{-10}$ & 6.82 & $3.14\,10^{-10}$ & $2.30\,10^{-10}$ & $7.69\,10^{-10}$ \\
 & 4 & 51206 & $4.22\,10^{-10}$ & 5.94 & $2.25\,10^{-12}$ & 6.84 & $2.82\,10^{-12}$ & $1.98\,10^{-12}$ & $7.03\,10^{-12}$ \\
\bottomrule
\end{tabular}
\end{table}

\begin{table}[htbp]
\centering
\caption{Poisson problem, Test 2: right-hand side $f(p)=p_1$ prescribed on the physical surface, so that $w$ enters the load vector at every quadrature point; degrees $4$ and $6$ (degree $2$ and the reference details are in Appendix~\ref{sm:test2}). Errors in $L^2(M)$ and the $H^1(M)$ seminorm with rates in parentheses; the red blood cell is in micrometers. The dumbbell reference is converged only to about $10^{-7}$, which limits its last rows.}
\label{tab:poisson2main}
\footnotesize
\setlength{\tabcolsep}{3pt}
\begin{tabular}{llrrll}
\toprule
geometry & $(d,r)$ & level & $\dim S_h$ & $L^2(M)$ (rate) & $H^1(M)$ (rate) \\
\midrule
sphere & $(4,0)$ & 2 & 2562 & $1.64\times10^{-6}$ (5.21) & $7.50\times10^{-5}$ (4.19) \\
 & $(4,0)$ & 3 & 10242 & $5.20\times10^{-8}$ (5.05) & $4.72\times10^{-6}$ (4.05) \\
 & $(6,1)$ & 2 & 3206 & $9.75\times10^{-9}$ (6.96) & $5.63\times10^{-7}$ (6.07) \\
\midrule
spheroid & $(4,0)$ & 2 & 2562 & $2.44\times10^{-6}$ (5.04) & $9.69\times10^{-5}$ (4.10) \\
 & $(4,0)$ & 3 & 10242 & $8.39\times10^{-8}$ (4.94) & $6.35\times10^{-6}$ (3.99) \\
 & $(6,1)$ & 2 & 3206 & $3.50\times10^{-8}$ (6.32) & $1.36\times10^{-6}$ (5.60) \\
\midrule
dumbbell & $(4,0)$ & 2 & 2562 & $1.45\times10^{-5}$ (4.30) & $3.75\times10^{-4}$ (3.44) \\
 & $(4,0)$ & 3 & 10242 & $4.05\times10^{-7}$ (5.23) & $2.32\times10^{-5}$ (4.07) \\
 & $(6,1)$ & 2 & 3206 & $8.46\times10^{-7}$ (4.95) & $2.79\times10^{-5}$ (4.23) \\
\midrule
red blood cell & $(4,0)$ & 2 & 2562 & $2.27\times10^{-4}$ (5.14) & $3.29\times10^{-3}$ (4.13) \\
 & $(4,0)$ & 3 & 10242 & $7.57\times10^{-6}$ (4.98) & $2.15\times10^{-4}$ (3.99) \\
 & $(6,1)$ & 2 & 3206 & $4.16\times10^{-6}$ (6.63) & $6.47\times10^{-5}$ (5.70) \\
\bottomrule
\end{tabular}
\end{table}

In Test 1 the $H^1(M)$ rates over the last refinement are $2.0$,
$4.0$ and $5.9$ for $d = 2, 4, 6$, the same to the digits shown across
the instances, and the $L^2(M)$ rates are $3.0$, $5.0$ and $6.8$, with
the $d=6$ errors reaching $2\times10^{-12}$ in $L^2(M)$ at $51206$
unknowns; the conformal factor enters the $L^2(M)$ errors as a
multiplicative constant within the bound $E$ of
Lemma~\ref{lem:transplant} and nothing else, and the prescribed factor
behaves like the embedded surfaces (Appendix~\ref{sm:test2}). In Test 2, where the weight is in the load vector, the rates on
the sphere, spheroid and red blood cell are again $d$ and $d+1$ to two
digits at $d = 2, 4$ and $5.6$ to $6.1$ in energy at $d = 6$; on the
dumbbell the finest $(6,1)$ row and one $(2,0)$ row fall short by an
amount within the accuracy of the spherical-harmonic reference, as
discussed in the appendix. The multiplier $\nu$ is below $10^{-13}$ and the projected
residual below $10^{-12}$ in every row.

\emph{Comparison with surface finite elements.} Test 1 was repeated on
the spheroid and the dumbbell with parametric surface finite elements
\cite{Dziuk88,Demlow09}: Dziuk's polyhedral $P_1$ method and
isoparametric $P_2$ and $P_3$ elements on the same icosahedral
triangulations pushed to $M$ by $\Phi$, every Lagrange node on the
exact surface, so that the geometric error is the smallest the
parametric method admits at that degree; errors are measured on the
discrete surface with the lifted exact solution.
Table~\ref{tab:compare} reports errors and the wall-clock time of one
Python implementation for both methods at the finest level run. At
equal degrees of freedom the $C^0$ spline of degree $2$ and the $P_2$
element are within a factor of two; the spline of degree $4$ has an
$L^2(M)$ error an order of magnitude below that of $P_3$ at $10^4$
unknowns; and the $C^1$ spline of degree $6$ reaches $3\times10^{-10}$
where $P_3$ reaches $10^{-6}$ with twice the unknowns. The finite element errors on
the dumbbell are two to three times those on the spheroid at the same
level, consistent with the larger geometric approximation error on the
more strongly curved surface, whereas the spline energy errors agree on the two instances to the
digits shown, as expected when the stiffness matrix is independent of
the metric and the map carries the geometry. In time, the per-instance
work of the spline at a given number of unknowns is comparable to $P_1$
and two to three times $P_2$, the difference being the order-$12$
product quadrature of the weighted load vector; the metric-independent
setup is amortized over instances. The ratios at the finest
configurations are in Table~\ref{tab:compare}.

\begin{table}[htbp]
\centering
\caption{Conformal spherical splines against parametric surface finite elements (Dziuk $P_1$, isoparametric $P_2$, $P_3$, nodes on the exact surface) for Poisson Test 1 on the spheroid and the dumbbell: finest level run for each degree, errors in $L^2(M)$ and the $H^1(M)$ seminorm with, in parentheses, the convergence rate over the last refinement, and wall-clock seconds. For the splines, ``setup is the geometry-independent work (smoothness matrix, null-space matrix, round-sphere stiffness), done once per triangulation and degree and reused for every surface; ``surface is the per-surface work (weight at the quadrature points, load and constraint vectors, bordered solve). For the FEM, ``setup is the isoparametric mesh and ``surface is assembly plus solve. Same machine, single thread, quadrature of order $12$ (splines) and $6$ (FEM).}
\label{tab:compare}
\footnotesize
\setlength{\tabcolsep}{3pt}
\begin{tabular}{llrllrr}
\toprule
geometry & method & dof & $L^2(M)$ & $H^1(M)$ & setup (s) & surface (s) \\
\midrule
spheroid & spline $(2,0)$ & 10242 & $3.39\,10^{-5}$ (3.00) & $2.92\,10^{-3}$ (2.00) & 2.3 & 12.4 \\
 & spline $(4,0)$ & 40962 & $4.45\,10^{-9}$ (5.01) & $6.70\,10^{-7}$ (4.01) & 3.1 & 14.4 \\
 & spline $(6,1)$ & 12806 & $3.14\,10^{-10}$ (6.79) & $2.49\,10^{-8}$ (5.90) & 1.5 & 5.5 \\
 & FEM $P_1$ & 10242 & $1.44\,10^{-3}$ (2.00) & $9.20\,10^{-2}$ (1.00) & 0.8 & 12.9 \\
 & FEM $P_2$ & 40962 & $4.77\,10^{-6}$ (3.00) & $9.29\,10^{-4}$ (2.00) & 1.1 & 15.1 \\
 & FEM $P_3$ & 23042 & $1.05\,10^{-6}$ (4.02) & $1.53\,10^{-4}$ (3.00) & 0.4 & 4.4 \\
\midrule
dumbbell & spline $(2,0)$ & 10242 & $2.17\,10^{-5}$ (3.00) & $2.92\,10^{-3}$ (2.00) & 2.2 & 24.2 \\
 & spline $(4,0)$ & 40962 & $2.78\,10^{-9}$ (5.01) & $6.70\,10^{-7}$ (4.01) & 3.2 & 25.6 \\
 & spline $(6,1)$ & 12806 & $2.30\,10^{-10}$ (6.80) & $2.49\,10^{-8}$ (5.90) & 1.4 & 8.4 \\
 & FEM $P_1$ & 10242 & $1.00\,10^{-3}$ (1.99) & $1.43\,10^{-1}$ (1.00) & 0.8 & 25.4 \\
 & FEM $P_2$ & 40962 & $3.26\,10^{-6}$ (3.06) & $2.55\,10^{-3}$ (1.99) & 1.2 & 27.7 \\
 & FEM $P_3$ & 23042 & $2.14\,10^{-6}$ (4.04) & $6.42\,10^{-4}$ (2.96) & 0.5 & 7.4 \\
\bottomrule
\end{tabular}
\end{table}

\subsection{The eigenproblem on the round sphere: the parity fingerprint}
\label{ssec:num-sphere}

Table~\ref{tab:sphere} reports the discrete spectra on the round sphere
($w \equiv 1$) against the exact eigenvalues $\ell(\ell+1)$ with
multiplicities $2\ell + 1$, split into the even-$\ell$ and odd-$\ell$
clusters of Proposition~\ref{prop:parity}. The prediction is realized
exactly: the even clusters are reproduced to at worst
$2.5 \times 10^{-13}$, independent of the mesh, at all three degrees,
while the odd clusters converge from above with measured rates $4.02$,
$8.35$, and $11.95$ against the expected $2d = 4, 8, 12$ of
Section~\ref{ssec:eigen}. At $d = 6$ the exact values $0, 6, 20, 42$
appear with multiplicities $1, 5, 9, 13$. At level $1$ the vertex
stars have diameter about $1.26$, violating $\abs{\mathrm{star}(v)} <
1$ in (A1), and the fingerprint and rates hold regardless.

\begin{table}[htbp]
\centering
\caption{Round sphere. Maximum eigenvalue errors over the even-$\ell$ and
odd-$\ell$ clusters (first $16$ eigenvalues for $d = 2, 4$; first $50$
for $d = 6$), and the measured convergence rate of the odd clusters
against the theoretical $2d$. The even clusters are exact up to
quadrature and roundoff at every mesh, as Proposition~\ref{prop:parity}
predicts.}
\label{tab:sphere}
\begin{tabular}{cccrccc}
\toprule
$d$ & $r$ & level & $\dim S^r_d$ & even-$\ell$ err & odd-$\ell$ err &
rate ($2d$) \\
\midrule
2 & 0 & 1 & 162  & $1.7\times 10^{-14}$ & $4.94\times 10^{-2}$ & \\
2 & 0 & 2 & 642  & $3.4\times 10^{-14}$ & $3.39\times 10^{-3}$ & 4.09 \\
2 & 0 & 3 & 2562 & $2.0\times 10^{-14}$ & $2.18\times 10^{-4}$ & 4.02 \\
\midrule
4 & 0 & 1 & 642  & $1.7\times 10^{-14}$ & $6.83\times 10^{-6}$ & \\
4 & 0 & 2 & 2562 & $1.8\times 10^{-14}$ & $2.88\times 10^{-8}$ & 8.35 \\
\midrule
6 & 1 & 1 & 806  & $1.5\times 10^{-12}$ & $4.29\times 10^{-6}$ & \\
6 & 1 & 2 & 3206 & $2.5\times 10^{-13}$ & $1.71\times 10^{-9}$ & 11.95 \\
\bottomrule
\end{tabular}
\end{table}

\subsection{The eigenproblem on the curved surfaces}
\label{ssec:num-spheroid}\label{ssec:num-rbc}\label{ssec:num-dumbbell}

For the spheroid the Mercator matching \eqref{eq:merc-match} reads
\begin{equation}\label{eq:mercator}
  \ln \tan(\theta/2) + G_0(\theta) = \ln \tan(\Theta/2),
  \qquad
  G_0(\theta) = \int_{\pi/2}^{\theta}
  \frac{\sqrt{\cos^2 t + c^2 \sin^2 t\,} - 1}{\sin t}\, dt,
\end{equation}
the isometric latitude of the Mercator projection of an ellipsoid, and $\int_{\sphere} w \dA$ agrees with the closed-form
area $2\pi + 2\pi c \arcsin(e)/e$, $e = \sqrt{1 - 1/c^2}$, to relative
error $8.4 \times 10^{-16}$. None of the three spectra has an
elementary closed form, so the reference is a dense Galerkin solve of
the same pencil in the real spherical-harmonic basis, in which the
zonal weight couples only equal azimuthal orders; its self-convergence
is given in Table~\ref{tab:curved}. The normalized first
eigenvalues $\mu_1 \operatorname{Area}/(4\pi)$ are $1.5514$, $0.780$
and $1.784$, all below the Hersch bound $2$; on the red blood cell the equatorial pair $\mu_{1,2} =
0.16715\,\mu\mathrm{m}^{-2}$ lies below the axial mode $\mu_3 =
0.23963$, so in-plane variation is cheaper than top-to-bottom variation
on the flat discocyte.

Table~\ref{tab:curved} gives the rates. On the spheroid
($\osc\lambda = 0.535$) they are $3.98$, $7.99$ and $10.62$ for $d =
2, 4, 6$, the last at an error of $2 \times 10^{-8}$ where the
quadrature of the weight is not yet negligible. On the red blood cell
($\osc\lambda = 0.373$) they are $3.98$, $8.02$ and $10.65$ with
terminal errors $1.3 \times 10^{-10}$ and $3.4 \times 10^{-11}$ at $d =
4, 6$. On the dumbbell ($\osc\lambda = 2.23$) the coarse errors are orders of
magnitude above the spheroid values at identical dimensions and the
rates climb toward $2d$ with refinement ($3.94$ at $d = 2$; $6.31$ then $7.00$ at $d = 4$;
$9.03$ at $d = 6$ from one pair of levels), consistent with the
expected order and a pre-asymptotic range that widens with
$\osc\lambda$.
One-sidedness holds to $6 \times 10^{-14}$ throughout.
Figure~\ref{fig:modes} shows computed eigenfunctions of the three
embedded instances.

\begin{figure}[htbp]
\centering
\begin{minipage}{0.24\textwidth}\centering
\includegraphics[width=\textwidth]{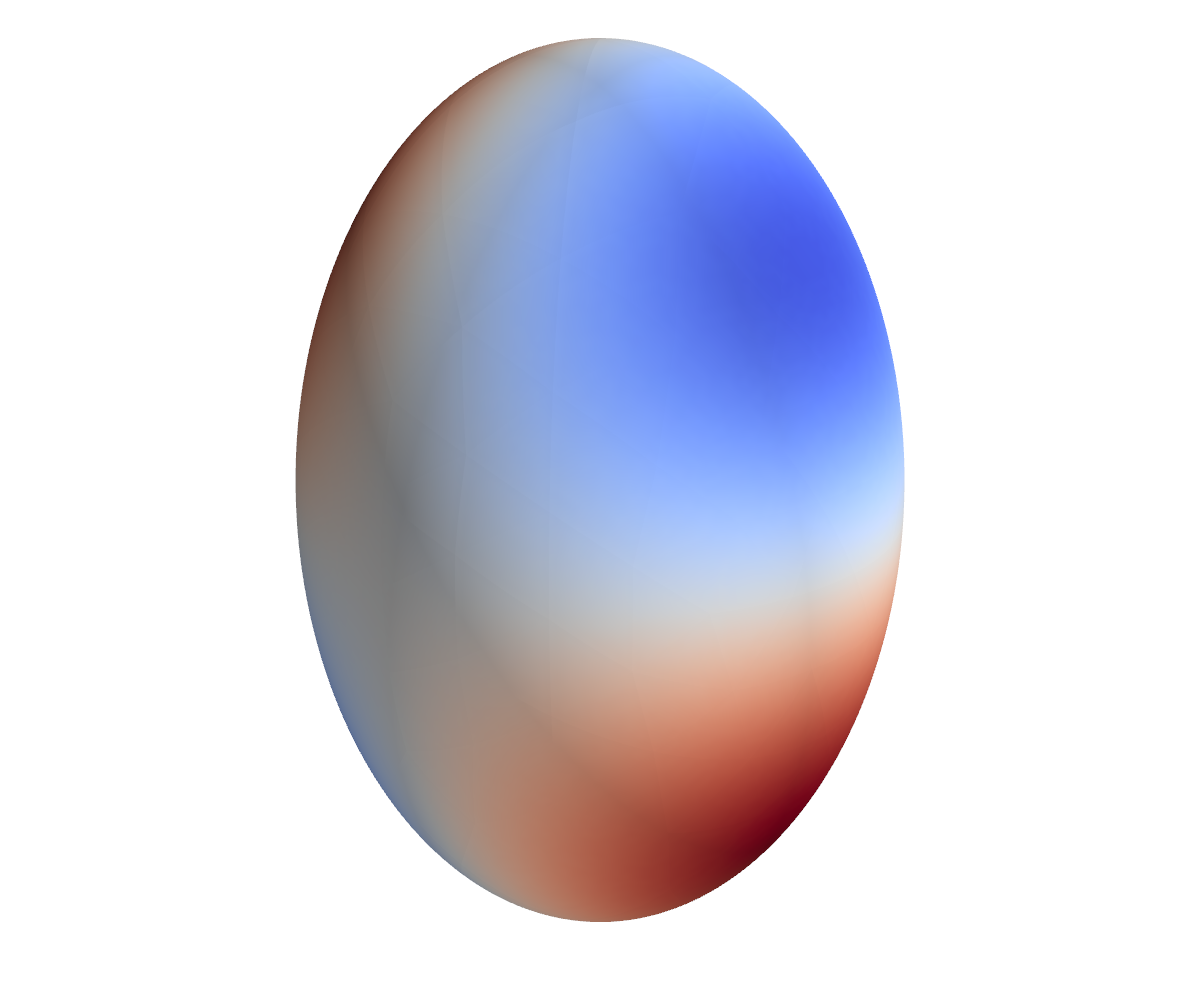}
\end{minipage}\hfill
\begin{minipage}{0.24\textwidth}\centering
\includegraphics[width=\textwidth]{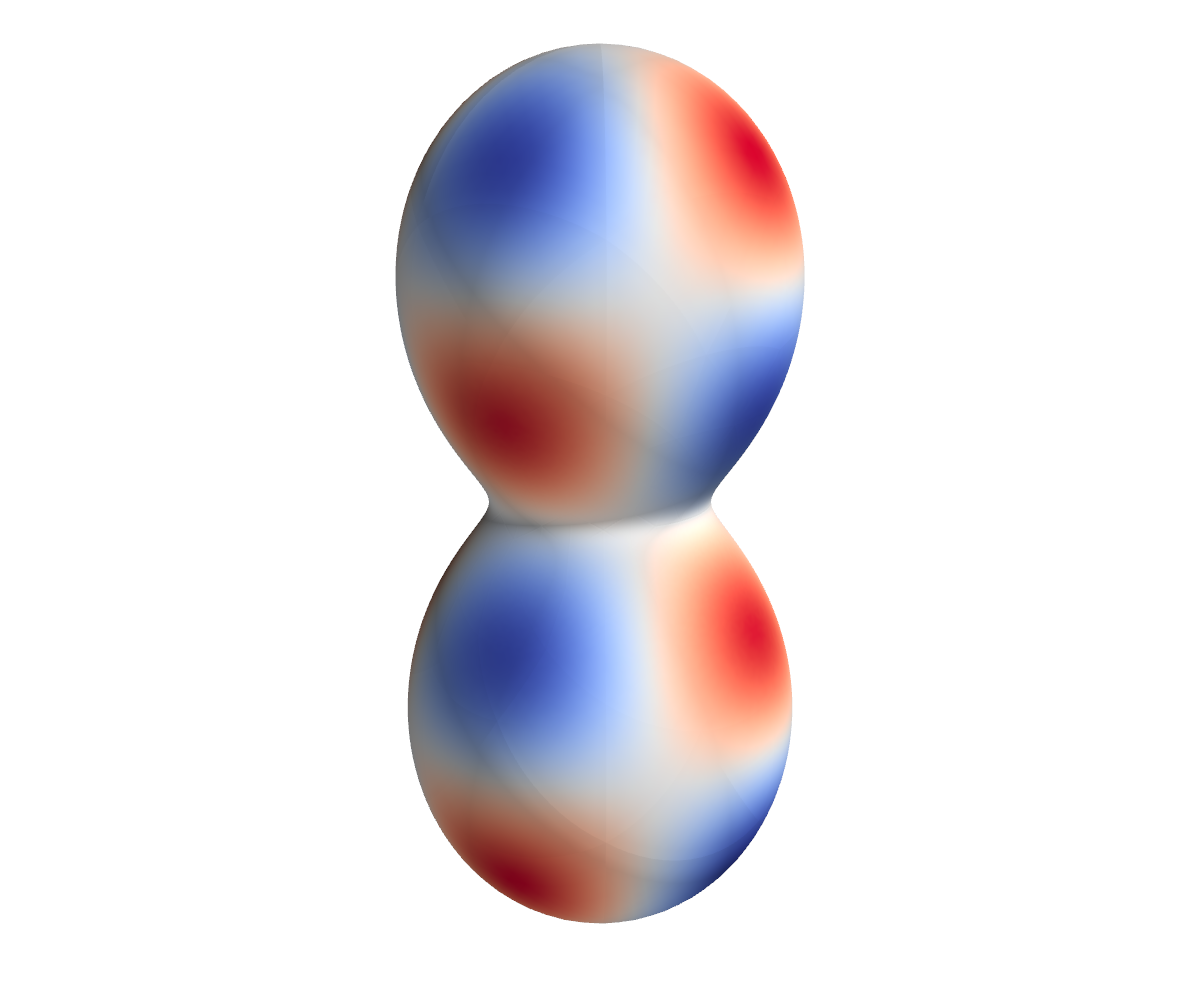}
\end{minipage}\hfill
\begin{minipage}{0.24\textwidth}\centering
\includegraphics[width=\textwidth]{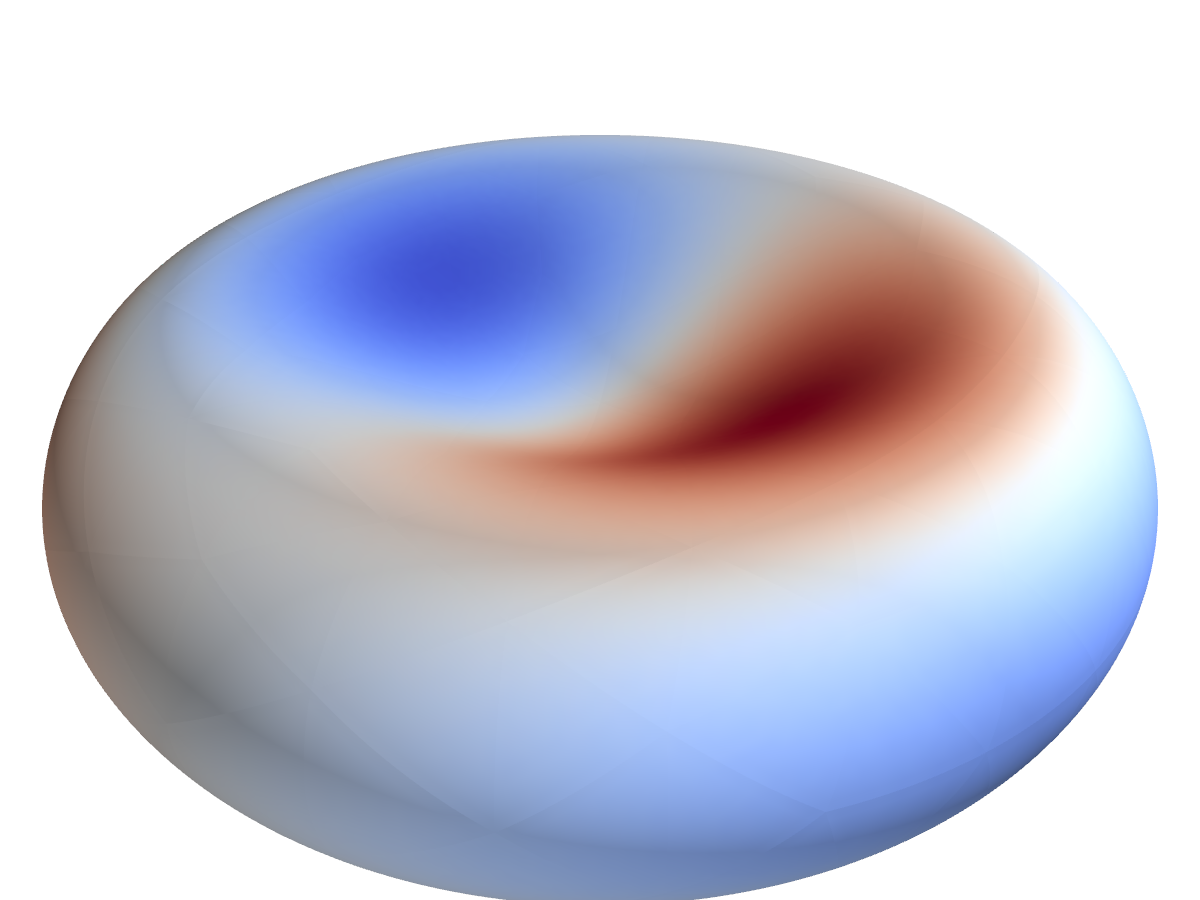}
\end{minipage}\hfill
\begin{minipage}{0.24\textwidth}\centering
\includegraphics[width=\textwidth]{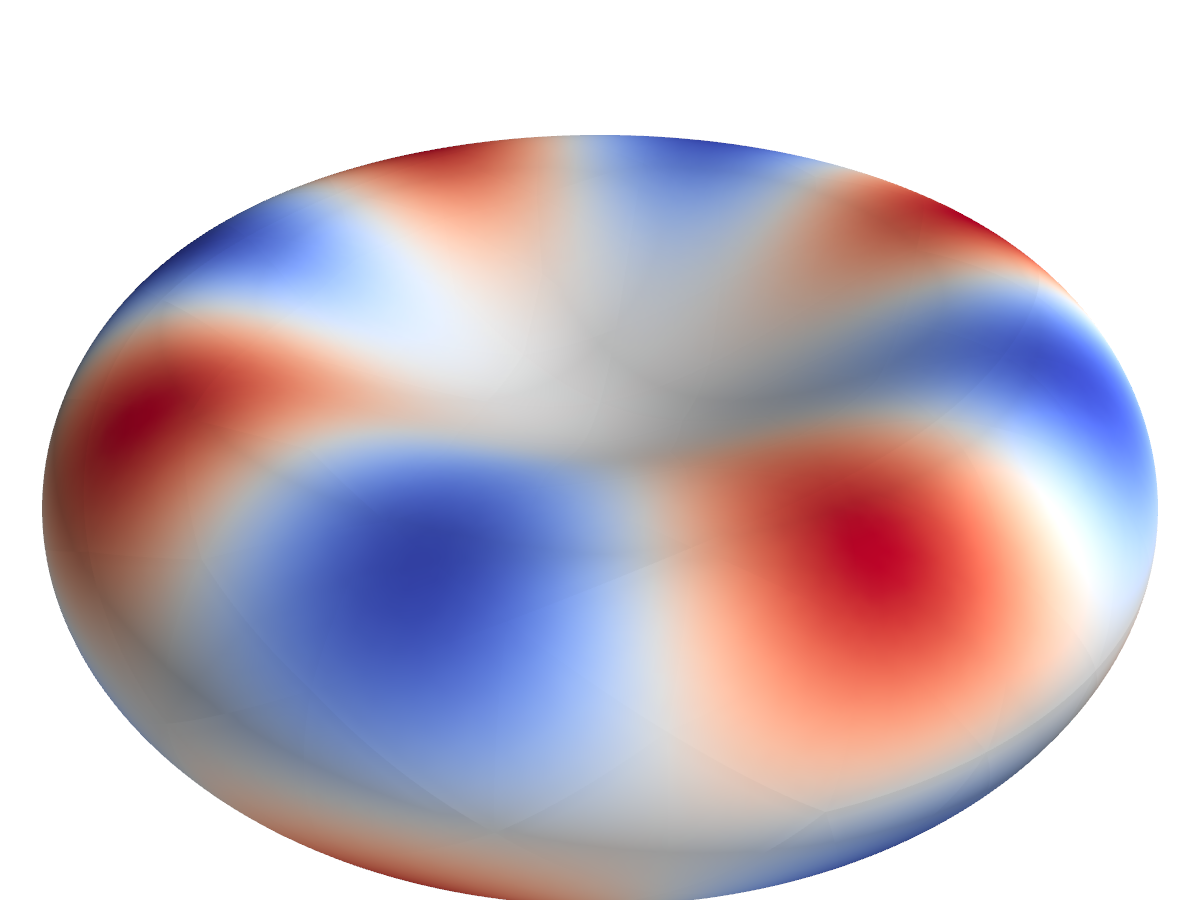}
\end{minipage}
\caption{Computed eigenfunctions at $d = 6$, $r = 1$, level $2$, drawn
on the embedded surfaces. From left: spheroid, seventh eigenfunction,
$\mu = 5.38$; dumbbell, $\mu = 57.62$, with azimuthal and axial nodal
lines crossing both bulbs; red blood cell, a mode localized in the
dimples, $\mu = 1.28\,\mu\mathrm{m}^{-2}$; red blood cell, an azimuthal
rim mode, $\mu = 2.66\,\mu\mathrm{m}^{-2}$. Sign and scale are
arbitrary.}
\label{fig:modes}
\end{figure}

\begin{table}[htbp]
\centering
\caption{Eigenvalues on the curved surfaces (levels $2$ and $3$; level $1$ in Appendix~\ref{sm:test2}). Maximum eigenvalue error of
the spline pencil against the spherical-harmonic reference (first $16$
eigenvalues; first $30$ at $d = 6$ on the spheroid) and measured rates
against $2d$. Reference self-convergence: $9.4 \times 10^{-13}$
(spheroid), $8.4 \times 10^{-8}$ (dumbbell), $2.2 \times 10^{-13}$
(red blood cell, in $\mu\mathrm{m}^{-2}$).}
\label{tab:curved}
\small
\begin{tabular}{lcccrcc}
\toprule
geometry & $d$ & $r$ & level & $\dim S^r_d$ & max error & rate ($2d$) \\
\midrule
spheroid & 2 & 0 & 2 & 642  & $4.47\times 10^{-3}$ & 4.27 \\
 & 2 & 0 & 3 & 2562 & $2.94\times 10^{-4}$ & 3.98 \\
 & 4 & 0 & 2 & 2562 & $3.63\times 10^{-7}$ & 7.99 \\
 & 6 & 1 & 2 & 3206 & $1.95\times 10^{-8}$ & 10.62 \\
\midrule
dumbbell & 2 & 0 & 2 & 642   & $1.28\times 10^{0}$  & 2.63 \\
 & 2 & 0 & 3 & 2562  & $8.68\times 10^{-2}$ & 3.94 \\
 & 4 & 0 & 2 & 2562  & $4.64\times 10^{-3}$ & 6.31 \\
 & 4 & 0 & 3 & 10242 & $3.88\times 10^{-5}$ & 7.00 \\
 & 6 & 1 & 2 & 3206  & $3.62\times 10^{-5}$ & 9.03 \\
\midrule
red blood cell & 2 & 0 & 2 & 642   & $5.84\times 10^{-4}$  & 4.26 \\
 & 2 & 0 & 3 & 2562  & $3.85\times 10^{-5}$  & 3.98 \\
 & 4 & 0 & 2 & 2562  & $3.10\times 10^{-8}$  & 7.70 \\
 & 4 & 0 & 3 & 10242 & $1.29\times 10^{-10}$ & 8.02 \\
 & 6 & 1 & 2 & 3206  & $3.35\times 10^{-11}$ & 10.65 \\
\bottomrule
\end{tabular}
\end{table}

\section{Discussion and conclusions}\label{sec:discussion}

For a conformal metric on the sphere the area element is the conformal
factor times the round one and the Dirichlet energy is the round
energy, so the stiffness matrix is the round-sphere matrix for every
such metric and the geometry is one scalar at the quadrature points of
the load vector, the constraint vector and the weighted mass matrix. By
uniformization every smooth closed genus-zero surface is such a metric,
and an embedded surface with a conformal parametrization, a surface of
revolution, a prescribed factor and a meshed surface are instances that
differ only in how the factor is obtained. The factor is checked against
an exact area where one exists before any equation is discretized, the
$C^r$ spline spaces serve every metric as they stand, and errors in the
metric are measured in the round one with no change of constants in
energy. The experiments show the rates $d+1$ in $L^2(M)$ and $d$ in
energy for the Poisson problem and $2d$ for the eigenvalues on all
instances, reproduction of the even-order eigenvalues of the round
metric to roundoff,
and, on
the dumbbell with $\osc\lambda = 2.23$, larger constants and a longer
pre-asymptotic range with the limiting orders unchanged.

Relative to surface finite elements
\cite{Dziuk88,DziukElliott13,BonitoDemlowNochetto20}, the method
replaces a sequence of physical meshes by a fixed spherical mesh plus
one scalar function. It is restricted to genus zero, needs the conformal
factor as input, and a metric with strong conformal concentration maps
large regions into small spherical patches that a uniform mesh
under-resolves; in return, when the factor is known exactly, the
geometry is low-dimensional, computed to the accuracy of the area
check, and, in
the cases run here, at degrees four and six the errors are smaller
than those of isoparametric cubic elements at comparable numbers of
unknowns.
Adaptive refinement driven by the conformal factor is the next step.
Two matters are deferred to separate papers: the constants
in \eqref{eq:H1rate} and in the eigenvalue estimate as functions of the
conformal factor and the position along the spectrum, and surfaces
given only as triangle meshes, for which the conformal factor must
itself be computed.

\section*{Acknowledgments}
Claude Fable 5.1 (Anthropic) assisted with drafting, editing, and code;
the mathematics, numerical results, final text, and source verification
are the author's responsibility.

\appendix
\section*{Appendix}
The appendix collects the implementation details and checks referred to
in the main text: the dependency count of the smoothness matrix, the
construction of the null-space matrix and the saddle-point realization,
the formulas evaluated by each step of the algorithm, the algebraic
checks that precede the experiments, and the second Poisson test.

\section{Dependencies among the smoothness rows}\label{sm:depcounts}

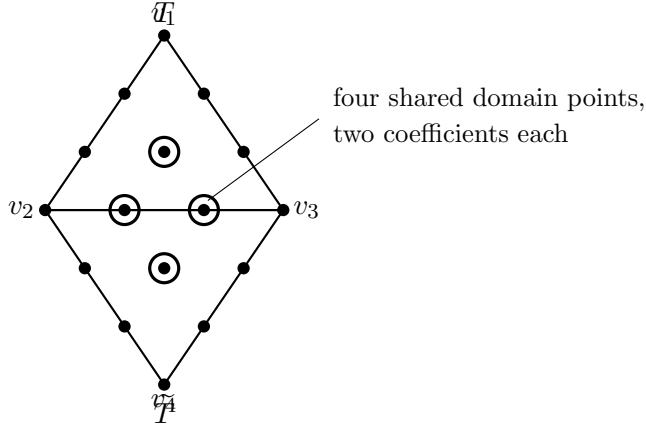
\begin{figure}[htbp]
\centering
\begin{tikzpicture}[scale=1.05]
\draw[thick] (0,0) -- (3,0) -- (1.5,2.2) -- cycle;
\draw[thick] (0,0) -- (3,0) -- (1.5,-2.2) -- cycle;
\node at (1.5,2.5) {$T$};
\node at (1.5,-2.5) {$\widetilde T$};
\node[left] at (0,0) {$v_2$};
\node[right] at (3,0) {$v_3$};
\node[above] at (1.5,2.2) {$v_1$};
\node[below] at (1.5,-2.2) {$v_4$};
\foreach \p in {(0,0),(1,0),(2,0),(3,0),(0.5,0.7333),(1.5,0.7333),
                (2.5,0.7333),(1.0,1.4667),(2.0,1.4667),(1.5,2.2)}
  \fill \p circle (2.2pt);
\foreach \p in {(0.5,-0.7333),(1.5,-0.7333),(2.5,-0.7333),
                (1.0,-1.4667),(2.0,-1.4667),(1.5,-2.2)}
  \fill \p circle (2.2pt);
\foreach \p in {(1,0),(2,0),(1.5,0.7333),(1.5,-0.7333)}
  \draw[very thick] \p circle (5.2pt);
\node[right, align=left] at (3.5,1.2) {\small four shared domain points,\\ \small two coefficients each};
\draw[thin] (3.45,1.15) -- (2.05,0.12);
\end{tikzpicture}
\caption{Domain points of two spherical triangles sharing an edge,
$d=3$. Every domain point on the shared edge carries two coefficients,
one from $T$ and one from $\widetilde T$; the $C^0$ rows
\eqref{eq:C0-rows} set their differences to zero. One $C^1$ row
\eqref{eq:C1-rows} ties the four circled coefficients through the
spherical barycentric coordinates $(\beta_1,\beta_2,\beta_3)$ of $v_4$
relative to $T$. The stencil is the planar one; the geometry of the
sphere enters only through the three numbers $\beta_m$.}
\label{fig:stencil}
\end{figure}

The rows of the smoothness matrix $J^{(r)}_0$ are linearly dependent,
since the conditions around a vertex are related. Two counts are elementary. The $C^0$ rows around a vertex of valence $n$ match the
$n$ coefficients at that vertex in a cycle, which has rank $n-1$, so the
$C^0$ rows carry exactly $N_V$ dependencies. The $C^1$ rows of the first
layer around a vertex match the first-order jets of the $n$ incident
polynomials, which by \eqref{eq:bern-deriv} are affine in the vertex
coefficient and the two adjacent edge coefficients; the jets form a
three-dimensional space, so these $n$ rows have rank $n-2$ and
contribute two further dependencies at every vertex whose incident
edges span the tangent plane. Section~\ref{ssec:checks} confirms
both counts on the meshes used.

The dimension $N-\rank J^{(r)}_0$ of the spline space can be compared with closed-form counts.
For $r=0$ the dimension is the
number of domain points of $\tri$, one per vertex, $d-1$ per edge and
$\binom{d-1}{2}$ per triangle,
\begin{equation}\label{eq:dimS0}
  \dim S^0_d(\tri)=N_V+(d-1)N_E+\binom{d-1}{2}N_T ,
\end{equation}
for every $d\ge1$. For $r=1$ and $d\ge5$, on a spherical triangulation
whose vertices each lie on at least three great circles, so that no
vertex is singular, Theorem~13.45 gives
\begin{equation}\label{eq:dimS1}
  \dim S^1_d(\tri)=\binom d2N_E-\Bigl(\binom{d+2}2-3\Bigr)N_V+2\binom{d+2}2
  =\binom{d-1}2N_T+6 ,
\end{equation}
the two forms being equal by \eqref{eq:euler}; the constant $6$ is three
times the Euler characteristic of the sphere. Both identities are
checked against the computed ranks in Appendix~\ref{sm:checks}.

\section{The null-space matrix, the saddle-point form, and residuals}
\label{sm:nullspace}

\paragraph{Construction of the null-space matrix}\label{ssec:nullspace}
The null-space matrix is built once per triangulation and degree by
sparse Gaussian elimination of the smoothness matrix $J$ to reduced row
echelon form, in a column order chosen for locality: the triangles are
visited breadth-first from a root triangle, and within a triangle the
domain points nearest the edge through which it was entered come first.
The elimination is rank revealing \cite[Ch.~5]{GolubVanLoan13}, with a
pivot tolerance of $10^{-8}$ relative to the unit entries of $J$, so
dependent rows, which occur
around every vertex, and any extra dimensions at singular vertices are
handled without case analysis. Let $D$ be the set of pivot columns and
$F$ the set of free columns, and let $\Pi$ be the permutation that
orders the columns as $(D,F)$. After elimination,
\begin{equation}\label{eq:rref}
  EJ\Pi=\begin{bmatrix}I&R\\0&0\end{bmatrix},\qquad
  Z=\Pi\begin{bmatrix}-R\\I\end{bmatrix},\qquad
  JZ=0,\qquad \abs F=N-\rank J ,
\end{equation}
where $E$ is the product of the row operations and the identity block
of $Z$ sits on the free columns. The free coefficients $F$ are a
determining set for $S^r_d(\tri)$: prescribing them determines all
pivot coefficients through $R$. The construction is combinatorial in
the mesh and metric only through the coordinates $\beta_m$ of
\eqref{eq:C1-rows}, so $Z$ is the same matrix for every surface $M$
represented on $\tri$, and in the experiments it is computed once per
$(\tri,d,r)$ and reused for all geometries. Figure~\ref{fig:pipeline}
shows how the objects fit together.

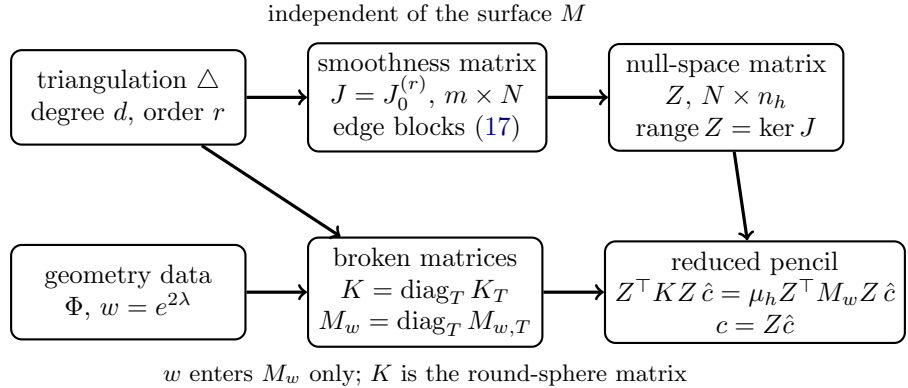
\begin{figure}[htbp]
\centering
\begin{tikzpicture}[
  box/.style={draw, thick, rounded corners, minimum height=1.25cm, minimum width=3.1cm, align=center, font=\small},
  note/.style={font=\footnotesize, align=center},
  arr/.style={->, very thick},
  node distance=1.3cm and 0.8cm]
\node[box] (tri) {triangulation $\tri$\\ degree $d$, order $r$};
\node[box, right=of tri] (J) {smoothness matrix\\ $J=J^{(r)}_0$, $m\times N$\\ edge blocks \eqref{eq:LS-smoothness}};
\node[box, right=of J] (Z) {null-space matrix\\ $Z$, $N\times n_h$\\ $\range Z=\ker J$};
\node[box, below=of tri] (w) {geometry data\\ $\Phi$, $w=e^{2\lambda}$};
\node[box, right=of w] (KM) {broken matrices\\ $K=\operatorname{diag}_TK_T$\\ $M_w=\operatorname{diag}_TM_{w,T}$};
\node[box, right=of KM] (pen) {reduced pencil\\ $Z^{\top}KZ\,\hat c=\mu_hZ^{\top}M_wZ\,\hat c$\\ $c=Z\hat c$};
\draw[arr] (tri) -- (J);
\draw[arr] (J) -- (Z);
\draw[arr] (tri) -- (KM);
\draw[arr] (w) -- (KM);
\draw[arr] (KM) -- (pen);
\draw[arr] (Z) -- (pen);
\node[note, above=0.08cm of J] {independent of the surface $M$};
\node[note, below=0.08cm of KM] {$w$ enters $M_w$ only; $K$ is the round-sphere matrix};
\end{tikzpicture}
\caption{The objects of the discretization and their dependencies. The
smoothness matrix $J$ and the null-space matrix $Z$ depend on the
triangulation, the degree and the smoothness order only. The broken
stiffness matrix $K$ is the round-sphere matrix for every genus-zero
surface. The surface enters through the weight $w$ in the broken
weighted mass matrix $M_w$ alone.}
\label{fig:pipeline}
\end{figure}

\paragraph{Rank tolerance and sparsity} The dimension of the
approximation space depends on the pivot tolerance of the elimination,
so its effect is tested. Table~\ref{tab:tolsparse} (top) repeats the
$C^1$, $d=6$ construction at level $1$ with tolerances $10^{-6}$ to
$10^{-12}$: the rank, the dimension and the first eigenvalues of the
round sphere agree to all printed digits, so the results are
insensitive to the tolerance over this range. Table~\ref{tab:tolsparse} (bottom)
gives the sparsity of $Z$ and of the reduced stiffness matrix
$Z^\top K Z$ for levels $1$ to $4$: the number of nonzeros per column
of $Z$ is $6$ and the number per row of $Z^\top K Z$ is $63$,
independent of the level over these four levels, so the reduction
generates no growth in fill, and
$Z^\top K Z$ has fewer nonzeros than the broken matrix $K$ ($4.0
\times 10^6$ at level $4$). The construction of $Z$ takes under two
seconds at level $4$ in the Python implementation.

\begin{table}[htbp]
\centering
\caption{$C^1$ splines of degree $6$. Top: rank of $J$, dimension
$N-\rank J$, and the first nonzero eigenvalue and the first eigenvalue
of the second cluster of the round sphere at level $1$, against the
pivot tolerance $\tau$ of the elimination. Bottom: size and sparsity
of the null-space matrix and of the reduced stiffness matrix by level.}
\label{tab:tolsparse}
\footnotesize
\begin{tabular}{lcccc}
\toprule
$\tau$ & $\rank J$ & $\dim S^1_6$ & $\mu_{1,h}$ & $\mu_{4,h}$ \\
\midrule
$10^{-6}$ & 1434 & 806 & 2.0000000008593 & 6.0000000000000 \\
$10^{-8}$ & 1434 & 806 & 2.0000000008593 & 6.0000000000000 \\
$10^{-10}$ & 1434 & 806 & 2.0000000008593 & 6.0000000000000 \\
$10^{-12}$ & 1434 & 806 & 2.0000000008593 & 6.0000000000000 \\
\bottomrule
\end{tabular}

\medskip
\begin{tabular}{rrrrrrrr}
\toprule
level & $N$ & $n_h$ & $\operatorname{nnz}(Z)$ & $\operatorname{nnz}(Z^\top K Z)$ & per col. of $Z$ & per row of $Z^\top KZ$ & $t_Z$ (s) \\
\midrule
1 & 2240 & 806 & 4744 & 49766 & 5.9 & 61.7 & 0.03 \\
2 & 8960 & 3206 & 19084 & 201030 & 6.0 & 62.7 & 0.12 \\
3 & 35840 & 12806 & 76724 & 805958 & 6.0 & 62.9 & 0.49 \\
4 & 143360 & 51206 & 308324 & 3225414 & 6.0 & 63.0 & 1.90 \\
\bottomrule
\end{tabular}
\end{table}

\paragraph{The saddle-point form and the broken residual}
With a full-row-rank $J_b$ spanning the row space of $J$ and Lagrange
multipliers $\lambda$ for the smoothness rows and $\nu$ for the mean
value, \eqref{eq:constrained-poisson} is also the symmetric indefinite
system with block matrix $[K, J_b^\top, M_w c_1;\ J_b, 0, 0;\
c_1^\top M_w, 0, 0]$, solved directly without any null-space matrix;
the two realizations agree to roundoff (Section~\ref{ssec:checks}).
For the solution $c$ of either, the broken residual $Kc - F + \nu M_w
c_1$ is orthogonal to $\ker J$, hence
\begin{equation}\label{eq:residual-multiplier}
  Kc - F + \nu M_w c_1 = -J^{\top}\lambda\qquad\text{for some }\lambda\in\R^m ,
\end{equation}
because $(\ker J)^\perp=\range J^{\top}$. The multiplier $\lambda$ is
not zero in general, since single-triangle test functions are not
conforming and their interelement flux terms do not cancel. The raw
residual is therefore $O(1)$ for the exact discrete solution and
vanishes only after projection by $Z^{\top}$ or subtraction of its
component in $\range J^{\top}$; every residual check in this paper is
evaluated in one of these two ways.

\section{The steps of the algorithm in detail}\label{sm:steps}

\paragraph{Step 1: compute and accept the conformal factor} The factor
is computed by \eqref{eq:w-coords} for a general conformal map and by
\eqref{eq:merc-gen}--\eqref{eq:w-stable} for a surface of revolution.
The equation \eqref{eq:merc-match} is solved for $\eta$ at each
quadrature colatitude $\Theta$ by a safeguarded Newton iteration with
bisection as fallback, which converges to roundoff in a few steps, and
the factor is evaluated in the half-angle form \eqref{eq:w-stable}.
Before assembly we record the relative area defect against the exact
area, the Hersch residual $\max_i \abs{\int_{\sphere} x_i w \dA}$, and
the oscillation $\osc\lambda = \tfrac12(\ln\max w - \ln\min w)$. The
first two decide whether $w$ is accepted; the third predicts where a
uniform spherical mesh will under-resolve, since a region of $M$ with
area fraction $\theta$ maps to a spherical region of area fraction of
the order $\theta/w$.

\paragraph{Step 3: the stiffness matrix} Quadrature on a
spherical triangle is the radial projection of a reference plane rule:
a flat node $u$ maps to $v = u/\norm{u}$ with weight
$2 A_{\mathrm{flat}}\, h_n\, \norm{u}^{-3}$ times the reference weight,
where $A_{\mathrm{flat}}$ is the area of the flat triangle and $h_n$
its distance from the origin; the factor $2 A_{\mathrm{flat}}\, h_n\,
\norm{u}^{-3}$ is the exact surface Jacobian of the radial projection. The element matrices $K_T$ use the tangential gradients
\eqref{eq:gradmatrix}, \eqref{eq:tangrad}; $K$ contains no $w$, and
$Z^\top K Z$ is formed and stored once per $(\tri, d, r)$. Both $K_T$
and $M_{w,T}$ can be assembled by the product-to-moment identities of
\cite{AinsworthAD11}, which reproduce direct quadrature to roundoff
(Section~\ref{ssec:checks}).

\paragraph{Step 4: the weighted data} On each triangle, at the
quadrature nodes $v_q$ with weights $W_q$, evaluate the Bernstein basis
$B^d_\alpha(v_q)$ by de Casteljau, the factor $w(v_q)$ from Step 1,
and $f(v_q) = f(\Phi(v_q))$; then
\begin{equation}\label{eq:loadquad}
\begin{aligned}
  (F_T)_\alpha &\approx \sum_q W_q\, w(v_q)\, f(v_q)\, B^d_\alpha(v_q),\\
  (M_{w,T} c_{1,T})_\alpha &\approx \sum_q W_q\, w(v_q)\,
  \Bigl(\sum_\beta (c_{1,T})_\beta B^d_\beta(v_q)\Bigr) B^d_\alpha(v_q),
\end{aligned}
\end{equation}
where the inner sum is the constant $1$ up to roundoff, $c_1$ being
the exact B-form of the constant. For
a manufactured $u$ the product $wf$ is
$-\Delta_{\sphere}u$ and the load vector does not see $w$; for
a right-hand side given on $M$ the product $w(v_q) f(\Phi(v_q))$ is
where the metric enters. The discrete compatibility $c_1^\top F$ is
the quadrature of $\int_M f\dA_M$ and is checked against zero. For the
eigenproblem the full $M_{w,T}$ is assembled by the same rule.

\paragraph{Step 5: solve} Form $g = Z^\top M_w c_1$ and $b = Z^\top F$
and solve \eqref{eq:reducedsystem}, a symmetric system of size $n_h+1$
with the pattern of $Z^\top K Z$ plus one dense row and column, by a
sparse direct solver; set $c = Z\hat c$. Record the multiplier $\nu$
and the projected residual $\norm{Z^\top(Kc - F + \nu M_w
c_1)}/\norm{Z^\top F}$, both at roundoff for a direct solve; the raw
residual $Kc - F$ is not a diagnostic, by
\eqref{eq:residual-multiplier}. For the eigenproblem, solve
\eqref{eq:matrixpencil} by shift-invert Lanczos \cite{LSY98} with a
small negative shift; the one-sidedness \eqref{eq:onesided} is
monitored in every experiment.

\paragraph{Cost} With $Z$ and $Z^\top K Z$ computed once, a new conformal factor
costs one pass over the quadrature points for $w$ and $f$ and
one factorization of the bordered matrix; a new right-hand side on the
same surface costs one pass and one triangular solve; the eigenproblem
costs one assembly of $M_w$ and one Lanczos run per surface. Of the
checks in Algorithm~\ref{alg:poisson}, a failure of the area identity
or the Hersch bound is repaired in Step 1, and a failure of
reproduction or of the parity fingerprint of
Proposition~\ref{prop:parity} in Steps 2 to 5.

\section{Algebraic checks}\label{sm:checks}

\begin{table}[htbp]
\centering
\caption{Geometry diagnostics of the four surfaces, recorded before any solve. The
area residual compares the represented area with an independent area
calculation.}
\label{tab:geomdiag}
\small
\setlength{\tabcolsep}{4pt}
\begin{tabular}{p{0.13\linewidth}p{0.25\linewidth}p{0.25\linewidth}p{0.25\linewidth}}
\toprule
geometry & map / data & area check & concentration \\
\midrule
sphere & exact & rel. $<3\times10^{-15}$ & $\osc\lambda=0$ \\
spheroid & analytic & rel. $8.4\times10^{-16}$ & $\osc\lambda=0.535$ \\
dumbbell & revolution & rel. $2.2\times10^{-15}$ & $\osc\lambda=2.23$ \\
RBC & generating curve & at roundoff & $\osc\lambda=0.373$ \\
\bottomrule
\end{tabular}
\end{table}

The Euler relations \eqref{eq:euler} hold on every mesh and the
quadrature area of the sphere is $4\pi$ to $3 \times 10^{-15}$. At
level $1$ ($N_V=42$, $N_E=120$, $N_T=80$) the smoothness matrix for
$d=2$, $r=0$ has $360$ rows and rank $318$, and for $d=6$, $r=1$ it
has $1560$ rows and rank $1434$; the $42$ and $126 = 42 + 84$
dependencies are one, respectively one plus two, per vertex, as
counted in Section~\ref{ssec:smoothness}. The dimensions $N - \rank J$
are $162$, $642$, $806$ at level $1$ and $642$, $2562$, $3206$ at
level $2$ for $S^0_2$, $S^0_4$, $S^1_6$, in agreement with
\eqref{eq:dimS0} and \eqref{eq:dimS1}; the sparse rank agrees with a
dense singular value decomposition, and the interpolated constant lies
in $\range Z$ to $10^{-14}$. The constant and the harmonic $xy$
interpolate at the domain points and satisfy every smoothness row to
$7 \times 10^{-14}$; the constant lies in the kernel of $Z^{\top}KZ$ to
$10^{-14}$ with mass $4\pi$ to $4 \times 10^{-15}$. For $xy$ with
eigenvalue $6$ the broken residual $Kc_{xy}-6M c_{xy}$ has relative
norm $1.0$ in every configuration while its projection by $Z^{\top}$
is below $6.7\times10^{-14}$, as \eqref{eq:residual-multiplier}
states. The reduced pencil \eqref{eq:matrixpencil} and the saddle-point
pencil \eqref{eq:saddlepencil}, both solved by shift-invert Arnoldi at
$\sigma=-1$, give the first sixteen eigenvalues of the round sphere at
level $1$ to $6\times10^{-14}$ at $d = 2, 4, 6$, so $Z$ spans exactly
$\ker J$. The quantity $\min_i (\mu_{i,h} - \mu_i)$ stayed above
$-3 \times 10^{-13}$ in every experiment, as \eqref{eq:onesided}
requires up to roundoff and reference error. Assembly through the
product-to-moment identities of \cite{AinsworthAD11} reproduces direct
quadrature to $10^{-15}$ and runs three to five times faster.

\section{Poisson Test 2 and the convergence plot of Test 1}\label{sm:test2}

In Test 2 of Section~\ref{ssec:num-poisson} the right-hand side is prescribed
on the physical surface, $f(p) = p_1$ made mean-free, so the load
vector sees the surface at every quadrature point. Its reference is a
dense Galerkin solve of the transplanted problem in the
spherical-harmonic sector $\overline P^1_\ell(\cos\Theta)\cos\phi$,
$\ell \le L = 80$, to which the zonal weight and the azimuthal order of
$f$ confine the problem; its coefficients are converged to $10^{-12}$
on the sphere, spheroid and red blood cell, and to about $10^{-7}$ on
the dumbbell, whose concentrated weight needs many more harmonics.
Table~\ref{tab:poisson2} reports the results. On the sphere, spheroid
and red blood cell the rates are $d$ and $d+1$ to two digits at
$d = 2, 4$ and $5.6$ to $6.1$ in energy at $d = 6$. On the dumbbell the
rates at $(2,0)$ level $3$ ($2.15$, $1.64$) and $(6,1)$ level $2$
($4.95$, $4.23$) fall short; the errors there, $7.7\times10^{-5}$ and
$8.5\times10^{-7}$, are within a factor of a few of the reference
accuracy, so the deficit is attributable to the reference, and the
$(4,0)$ sequence on the same surface shows $5.23$ and $4.07$.
Figure~\ref{fig:poisson-conv} plots Test 1 against the mesh size.

\begin{table}[htbp]
\centering
\caption{Poisson problem, Test 2: right-hand side $f(p) = p_1$, the first Cartesian coordinate of the point $p \in M$, so that $w$ enters the load vector; reference solution from a spherical-harmonic Galerkin solve with $L = 80$, converged to $10^{-12}$ in the coefficients on the sphere, spheroid and red blood cell and to about $10^{-7}$ on the dumbbell. Errors in $L^2(M)$ and in the $H^1(M)$ seminorm; rates against the geodesic mesh size. The red blood cell is in physical units (micrometers), so its errors carry those units.}
\label{tab:poisson2}
\footnotesize
\setlength{\tabcolsep}{4pt}
\begin{tabular}{llrrllll}
\toprule
geometry & $(d,r)$ & level & $\dim S_h$ & $L^2(M)$ error & rate & $H^1(M)$ error & rate \\
\midrule
sphere & $(2,0)$ & 1 & 162 & $4.04\times10^{-3}$ &  & $4.88\times10^{-2}$ &  \\
 & $(2,0)$ & 2 & 642 & $5.06\times10^{-4}$ & 3.17 & $1.23\times10^{-2}$ & 2.10 \\
 & $(2,0)$ & 3 & 2562 & $6.34\times10^{-5}$ & 3.04 & $3.10\times10^{-3}$ & 2.02 \\
 & $(4,0)$ & 1 & 642 & $4.97\times10^{-5}$ &  & $1.17\times10^{-3}$ &  \\
 & $(4,0)$ & 2 & 2562 & $1.64\times10^{-6}$ & 5.21 & $7.50\times10^{-5}$ & 4.19 \\
 & $(4,0)$ & 3 & 10242 & $5.20\times10^{-8}$ & 5.05 & $4.72\times10^{-6}$ & 4.05 \\
 & $(6,1)$ & 1 & 806 & $9.30\times10^{-7}$ &  & $3.00\times10^{-5}$ &  \\
 & $(6,1)$ & 2 & 3206 & $9.75\times10^{-9}$ & 6.96 & $5.63\times10^{-7}$ & 6.07 \\
\midrule
spheroid & $(2,0)$ & 1 & 162 & $6.36\times10^{-3}$ &  & $6.52\times10^{-2}$ &  \\
 & $(2,0)$ & 2 & 642 & $7.87\times10^{-4}$ & 3.19 & $1.66\times10^{-2}$ & 2.09 \\
 & $(2,0)$ & 3 & 2562 & $9.90\times10^{-5}$ & 3.04 & $4.17\times10^{-3}$ & 2.02 \\
 & $(4,0)$ & 1 & 642 & $6.65\times10^{-5}$ &  & $1.43\times10^{-3}$ &  \\
 & $(4,0)$ & 2 & 2562 & $2.44\times10^{-6}$ & 5.04 & $9.69\times10^{-5}$ & 4.10 \\
 & $(4,0)$ & 3 & 10242 & $8.39\times10^{-8}$ & 4.94 & $6.35\times10^{-6}$ & 3.99 \\
 & $(6,1)$ & 1 & 806 & $2.20\times10^{-6}$ &  & $5.33\times10^{-5}$ &  \\
 & $(6,1)$ & 2 & 3206 & $3.50\times10^{-8}$ & 6.32 & $1.36\times10^{-6}$ & 5.60 \\
\midrule
dumbbell & $(2,0)$ & 1 & 162 & $3.70\times10^{-3}$ &  & $3.00\times10^{-2}$ &  \\
 & $(2,0)$ & 2 & 642 & $3.33\times10^{-4}$ & 3.68 & $7.63\times10^{-3}$ & 2.09 \\
 & $(2,0)$ & 3 & 2562 & $7.68\times10^{-5}$ & 2.15 & $2.49\times10^{-3}$ & 1.64 \\
 & $(4,0)$ & 1 & 642 & $2.41\times10^{-4}$ &  & $3.56\times10^{-3}$ &  \\
 & $(4,0)$ & 2 & 2562 & $1.45\times10^{-5}$ & 4.30 & $3.75\times10^{-4}$ & 3.44 \\
 & $(4,0)$ & 3 & 10242 & $4.05\times10^{-7}$ & 5.23 & $2.32\times10^{-5}$ & 4.07 \\
 & $(6,1)$ & 1 & 806 & $2.17\times10^{-5}$ &  & $4.44\times10^{-4}$ &  \\
 & $(6,1)$ & 2 & 3206 & $8.46\times10^{-7}$ & 4.95 & $2.79\times10^{-5}$ & 4.23 \\
\midrule
red blood cell & $(2,0)$ & 1 & 162 & $5.47\times10^{-1}$ &  & $2.11\times10^{0}$ &  \\
 & $(2,0)$ & 2 & 642 & $6.80\times10^{-2}$ & 3.18 & $5.38\times10^{-1}$ & 2.09 \\
 & $(2,0)$ & 3 & 2562 & $8.53\times10^{-3}$ & 3.04 & $1.35\times10^{-1}$ & 2.02 \\
 & $(4,0)$ & 1 & 642 & $6.55\times10^{-3}$ &  & $4.91\times10^{-2}$ &  \\
 & $(4,0)$ & 2 & 2562 & $2.27\times10^{-4}$ & 5.14 & $3.29\times10^{-3}$ & 4.13 \\
 & $(4,0)$ & 3 & 10242 & $7.57\times10^{-6}$ & 4.98 & $2.15\times10^{-4}$ & 3.99 \\
 & $(6,1)$ & 1 & 806 & $3.19\times10^{-4}$ &  & $2.71\times10^{-3}$ &  \\
 & $(6,1)$ & 2 & 3206 & $4.16\times10^{-6}$ & 6.63 & $6.47\times10^{-5}$ & 5.70 \\
\bottomrule
\end{tabular}
\end{table}

Table~\ref{tab:prescribed} gives Test 1 for the prescribed factor
$\lambda = 0.6\,(x_1^2 - x_2^2) - 0.4\,x_3^2$, an instance with no
embedded surface: the metric is defined by $\lambda$ alone, the
Hersch condition holds by symmetry, and the same computation applies
with $w = e^{2\lambda}$ evaluated at the quadrature points. The rates
are those of the embedded instances.

\begin{table}[htbp]
\centering
\caption{Poisson Test 1 for the prescribed factor $\lambda = 0.6\,(x_1^2 - x_2^2) - 0.4\,x_3^2$ with no embedding ($\osc\lambda = 1.2$). Errors in $L^2(M)$ and the $H^1(M)$ seminorm with rates against the geodesic mesh size.}
\label{tab:prescribed}
\footnotesize
\begin{tabular}{llrllll}
\toprule
$(d,r)$ & level & $\dim S_h$ & $L^2(M)$ & rate & $H^1(M)$ & rate \\
\midrule
$(2,0)$ & 1 & 162 & $1.26\times10^{-2}$ &  & $1.75\times10^{-1}$ &  \\
 & 2 & 642 & $1.72\times10^{-3}$ & 3.04 & $4.59\times10^{-2}$ & 2.04 \\
 & 3 & 2562 & $2.20\times10^{-4}$ & 3.01 & $1.16\times10^{-2}$ & 2.01 \\
 & 4 & 10242 & $2.77\times10^{-5}$ & 3.00 & $2.92\times10^{-3}$ & 2.00 \\
\midrule
$(4,0)$ & 1 & 642 & $1.12\times10^{-4}$ &  & $2.64\times10^{-3}$ &  \\
 & 2 & 2562 & $3.68\times10^{-6}$ & 5.22 & $1.70\times10^{-4}$ & 4.19 \\
 & 3 & 10242 & $1.17\times10^{-7}$ & 5.05 & $1.07\times10^{-5}$ & 4.05 \\
 & 4 & 40962 & $3.66\times10^{-9}$ & 5.01 & $6.70\times10^{-7}$ & 4.01 \\
\midrule
$(6,1)$ & 1 & 806 & $2.32\times10^{-6}$ &  & $7.34\times10^{-5}$ &  \\
 & 2 & 3206 & $2.49\times10^{-8}$ & 6.92 & $1.40\times10^{-6}$ & 6.05 \\
 & 3 & 12806 & $2.42\times10^{-10}$ & 6.79 & $2.49\times10^{-8}$ & 5.90 \\
 & 4 & 51206 & $2.19\times10^{-12}$ & 6.81 & $4.22\times10^{-10}$ & 5.91 \\
\bottomrule
\end{tabular}
\end{table}

\begin{figure}[htbp]
\centering
\includegraphics[width=\textwidth]{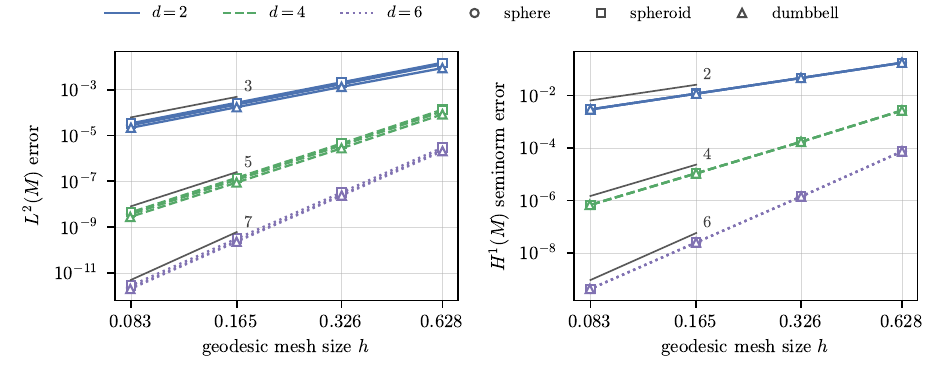}
\caption{Poisson problem, Test 1, icosahedral levels $1$ to $4$. Left:
error in $L^2(M)$; right: error in the $H^1(M)$ seminorm, against the
geodesic mesh size, for $d = 2, 4, 6$ (line style) on the sphere, the
spheroid and the dumbbell (marker). The short gray segments have the
slopes $d+1$ (left) and $d$ (right). The three geometries coincide in
energy because the load vector of Test 1 does not contain the conformal
factor, and are within a factor $E$ of each other in $L^2(M)$. The red
blood cell, in micrometers, is omitted from the plot.}
\label{fig:poisson-conv}
\end{figure}

\bibliographystyle{siamplain}
\bibliography{refs}

@article{AinsworthAD11,
  author = {M.~Ainsworth and G.~Andriamaro and O.~Davydov},
  title = {{B}ernstein--{B}\'ezier finite elements of arbitrary order and optimal assembly procedures},
  journal = {SIAM J. Sci. Comput.},
  volume = {33},
  year = {2011},
  pages = {3087--3109}
}

@incollection{AlfNS95, author={P. Alfeld and M. Neamtu and L. L. Schumaker}, title={Circular {B}ernstein--{B}\'ezier polynomials}, booktitle={Mathematical Methods for Curves and Surfaces}, editor={M. D{\ae}hlen and T. Lyche and L. L. Schumaker}, publisher={Vanderbilt University Press}, address={Nashville}, year={1995}, pages={11--20}}

@article{AlfNS96a,
  author = {P.~Alfeld and M.~Neamtu and L.~L.~Schumaker},
  title = {Fitting scattered data on sphere-like surfaces using spherical splines},
  journal = {J. Comput. Appl. Math.},
  volume = {73},
  year = {1996},
  pages = {5--43}
}

@article{AlfNS96b,
  author = {P.~Alfeld and M.~Neamtu and L.~L.~Schumaker},
  title = {{B}ernstein--{B}\'ezier polynomials on spheres and sphere-like surfaces},
  journal = {Comput. Aided Geom. Design},
  volume = {13},
  year = {1996},
  pages = {333--349}
}

@article{AlfNS96c,
  author = {P.~Alfeld and M.~Neamtu and L.~L.~Schumaker},
  title = {Dimension and local bases of homogeneous spline spaces},
  journal = {SIAM J. Math. Anal.},
  volume = {27},
  year = {1996},
  pages = {1482--1501}
}

@book{AtkinsonHan12, author={K. Atkinson and W. Han}, title={{Spherical Harmonics and Approximations on the Unit Sphere: An Introduction}}, series={Lecture Notes in Mathematics}, volume={2044}, publisher={Springer}, address={Berlin}, year={2012}}

@incollection{BabuskaOsborn91, author={I. Babu\v{s}ka and J. Osborn}, title={Eigenvalue problems}, booktitle={Handbook of Numerical Analysis, Vol.~II}, editor={P. G. Ciarlet and J. L. Lions}, publisher={North-Holland}, address={Amsterdam}, year={1991}, pages={641--787}}

@article{BenziGolubLiesen05, author={M. Benzi and G. H. Golub and J. Liesen}, title={Numerical solution of saddle point problems}, journal={Acta Numer.}, volume={14}, year={2005}, pages={1--137}}

@article{BonitoDemlow19,
  author = {A.~Bonito and A.~Demlow},
  title = {A posteriori error estimates for the {L}aplace--{B}eltrami operator on parametric $C^2$ surfaces},
  journal = {SIAM J. Numer. Anal.},
  volume = {57},
  year = {2019},
  pages = {973--996}
}

@incollection{BonitoDemlowNochetto20, author={A. Bonito and A. Demlow and R. H. Nochetto}, title={Finite element methods for the {L}aplace--{B}eltrami operator}, booktitle={Geometric Partial Differential Equations, Part I}, series={Handbook of Numerical Analysis}, volume={21}, publisher={Elsevier}, address={Amsterdam}, year={2020}, pages={1--103}}

@article{BonitoDemlowOwen18,
  author = {A.~Bonito and A.~Demlow and J.~Owen},
  title = {A priori error estimates for finite element approximations to eigenvalues and eigenfunctions of the {L}aplace--{B}eltrami operator},
  journal = {SIAM J. Numer. Anal.},
  volume = {56},
  year = {2018},
  pages = {2963--2988}
}

@article{Boffi10,
  author = {D.~Boffi},
  title = {Finite element approximation of eigenvalue problems},
  journal = {Acta Numer.},
  volume = {19},
  year = {2010},
  pages = {1--120}
}

@incollection{BarL05, author={V. Baramidze and M.-J. Lai}, title={Error bounds for minimal energy interpolatory spherical splines}, booktitle={Approximation Theory XI: Gatlinburg 2004}, editor={C. K. Chui and M. Neamtu and L. L. Schumaker}, publisher={Nashboro Press}, address={Brentwood, TN}, year={2005}, pages={25--50}}

@incollection{BarLai06, author={V. Baramidze and M.-J. Lai}, title={Spherical spline solution to a {PDE} on the sphere}, booktitle={Wavelets and Splines: Athens 2005}, editor={G. Chen and M.-J. Lai}, publisher={Nashboro Press}, address={Brentwood, TN}, year={2006}, pages={75--92}}

@article{BarLai11,
  author = {V.~Baramidze and M.-J.~Lai},
  title = {Convergence of discrete and penalized least squares spherical splines},
  journal = {J. Approx. Theory},
  volume = {163},
  year = {2011},
  pages = {1091--1106}
}

@article{BarLai18,
  author = {V.~Baramidze and M.~J.~Lai},
  title = {Nonnegative data interpolation by spherical splines},
  journal = {J. Comput. Appl. Math.},
  volume = {342},
  year = {2018},
  pages = {463--477}
}

@article{BarLaiShum06,
  author = {V.~Baramidze and M.-J.~Lai and C.~K.~Shum},
  title = {Spherical splines for data interpolation and fitting},
  journal = {SIAM J. Sci. Comput.},
  volume = {28},
  year = {2006},
  pages = {241--259}
}

@book{BrennerScott08, author={S. C. Brenner and L. R. Scott}, title={{The Mathematical Theory of Finite Element Methods}}, edition={3rd}, publisher={Springer}, address={New York}, year={2008}}

@book{Chavel84, author={I. Chavel}, title={{Eigenvalues in {R}iemannian Geometry}}, publisher={Academic Press}, address={Orlando}, year={1984}}

@article{Demlow09,
  author = {A.~Demlow},
  title = {Higher-order finite element methods and pointwise error estimates for elliptic problems on surfaces},
  journal = {SIAM J. Numer. Anal.},
  volume = {47},
  year = {2009},
  pages = {805--827}
}

@article{DemlowDziuk07,
  author = {A.~Demlow and G.~Dziuk},
  title = {An adaptive finite element method for the {L}aplace--{B}eltrami operator on implicitly defined surfaces},
  journal = {SIAM J. Numer. Anal.},
  volume = {45},
  year = {2007},
  pages = {421--442}
}

@incollection{Dziuk88, author={G. Dziuk}, title={Finite elements for the {B}eltrami operator on arbitrary surfaces}, booktitle={Partial Differential Equations and Calculus of Variations}, series={Lecture Notes in Mathematics}, volume={1357}, publisher={Springer}, address={Berlin}, year={1988}, pages={142--155}}

@article{DziukElliott13,
  author = {G.~Dziuk and C.~M.~Elliott},
  title = {Finite element methods for surface {PDEs}},
  journal = {Acta Numer.},
  volume = {22},
  year = {2013},
  pages = {289--396}
}

@article{EvansFung72,
  author = {E.~Evans and Y.-C.~Fung},
  title = {Improved measurements of the erythrocyte geometry},
  journal = {Microvasc. Res.},
  volume = {4},
  year = {1972},
  pages = {335--347}
}

@book{GolubVanLoan13, author={G. H. Golub and C. F. Van Loan}, title={{Matrix Computations}}, edition={4th}, publisher={Johns Hopkins University Press}, address={Baltimore}, year={2013}}

@inproceedings{GuYau03, author={X. Gu and S.-T. Yau}, title={Global conformal surface parameterization}, booktitle={Proceedings of the Eurographics/ACM SIGGRAPH Symposium on Geometry Processing}, year={2003}, pages={127--137}}

@article{Hersch70, author={J. Hersch}, title={Quatre propri\'et\'es isop\'erim\'etriques de membranes sph\'eriques homog\`enes}, journal={C. R. Acad. Sci. Paris S\'er. A}, volume={270}, year={1970}, pages={1645--1648}}

@article{KNPP21,
  author = {M.~Karpukhin and N.~Nadirashvili and A.~Penskoi and I.~Polterovich},
  title = {An isoperimetric inequality for {L}aplace eigenvalues on the sphere},
  journal = {J. Differential Geom.},
  volume = {118},
  year = {2021},
  pages = {313--333}
}

@article{KSBC12,
  author = {M.~Kazhdan and J.~Solomon and M.~Ben-Chen},
  title = {Can mean-curvature flow be modified to be non-singular?},
  journal = {Comput. Graph. Forum},
  volume = {31},
  year = {2012},
  pages = {1745--1754}
}

@book{LaiSchumaker07, author={M.-J. Lai and L. L. Schumaker}, title={{Spline Functions on Triangulations}}, publisher={Cambridge University Press}, address={Cambridge}, year={2007}}

@book{LSY98, author={R. B. Lehoucq and D. C. Sorensen and C. Yang}, title={{{ARPACK} Users' Guide: Solution of Large-Scale Eigenvalue Problems with Implicitly Restarted {A}rnoldi Methods}}, publisher={SIAM}, address={Philadelphia}, year={1998}}

@article{LiYau82,
  author = {P.~Li and S.-T.~Yau},
  title = {A new conformal invariant and its applications to the {W}illmore conjecture and the first eigenvalue of compact surfaces},
  journal = {Invent. Math.},
  volume = {69},
  year = {1982},
  pages = {269--291}
}

@article{NeamtuSchumaker04,
  author = {M.~Neamtu and L.~L.~Schumaker},
  title = {On the approximation order of splines on spherical triangulations},
  journal = {Adv. Comput. Math.},
  volume = {21},
  year = {2004},
  pages = {3--20}
}

@book{Taylor11, author={M. E. Taylor}, title={{Partial Differential Equations I: Basic Theory}}, edition={2nd}, publisher={Springer}, address={New York}, year={2011}}

@incollection{AwanouLaiWenston06, author={G. Awanou and M.-J. Lai and P. Wenston}, title={The multivariate spline method for scattered data fitting and numerical solution of partial differential equations}, booktitle={Wavelets and Splines: Athens 2005}, editor={G. Chen and M.-J. Lai}, publisher={Nashboro Press}, address={Brentwood, TN}, year={2006}, pages={24--75}}

@book{Jost06, author={J. Jost}, title={{Compact {R}iemann Surfaces: An Introduction to Contemporary Mathematics}}, edition={3rd}, series={Universitext}, publisher={Springer}, address={Berlin}, year={2006}}

@unpublished{Kapita26BCC, author={S. Kapita}, title={Bernstein constraint complexes for multivariate splines on triangulated surfaces}, note={Preprint}, year={2026}}

@article{LaiLee22, author={M.-J. Lai and J. Lee}, title={A multivariate spline based collocation method for numerical solution of partial differential equations}, journal={SIAM J. Numer. Anal.}, volume={60}, year={2022}, pages={2405--2434}}

@article{LaiLee23, author={M.-J. Lai and J. Lee}, title={Trivariate spline collocation methods for numerical solution to {3D} {M}onge--{A}mp\`ere equation}, journal={J. Sci. Comput.}, volume={95}, year={2023}, pages={56}}

\end{document}